\documentclass[a4paper]{article}

\usepackage[english]{babel}
\usepackage[T1]{fontenc}
\usepackage{amssymb}
\usepackage{amsthm}
\usepackage{makeidx}
\usepackage{color}
\usepackage{fullpage}
\usepackage{tikz}
\usepackage{mathtools}
\usepackage[all]{xy}
\usepackage{comment}
\usepackage{authblk}
\usepackage[backend=biber,style=numeric,natbib=true,maxbibnames=99]{biblatex}

\usepackage[a4paper,top=3cm,bottom=2cm,left=3cm,right=3cm,marginparwidth=1.75cm]{geometry}

\usepackage{amsmath}
\usepackage{graphicx}
\usepackage[colorinlistoftodos]{todonotes}
\usepackage[colorlinks=true, allcolors=blue]{hyperref}
\usepackage[all]{xy}

\newtheorem{defi}{Definition}[section]
\newtheorem{example}[defi]{Example}
\newtheorem{teo}[defi]{Theorem}
\newtheorem{coro}[defi]{Corollary}
\newtheorem{lema}[defi]{Lemma}

\newtheorem{pro}[defi]{Proposition}
\newtheorem{rmk}[defi]{Remark}

\providecommand{\keywords}[1]{\textbf{\textit{Keywords---}} #1}

\providecommand{\MSC}[1]{\textbf{\textit{MSC---}} #1}

\title{A generalization of Kummer theory to \\ Hopf-Galois extensions}
\author[1,2]{Daniel Gil-Muñoz}
\date{}

\affil[1]{\footnotesize Charles University, Faculty of Mathematics and Physics, Department of Algebra, Sokolovska 83, 18600 Praha 8, Czech Republic}
\affil[2]{\footnotesize Institut de Matemàtica, Universitat de Barcelona, Gran Via de les Corts Catalanes, 585, 08007 Barcelona, Spain}

\begin{document}
\maketitle

\begin{abstract}
We introduce a condition for Hopf-Galois extensions that generalizes the notion of Kummer Galois extension. Namely, an $H$-Galois extension $L/K$ is $H$-Kummer if $L$ can be generated by adjoining to $K$ a finite set $S$ of eigenvectors for the action of the Hopf algebra $H$ on $L$. This extends the classical Kummer condition for the classical Galois structure. With this new perspective, we shall characterize a class of $H$-Kummer extensions $L/K$ as radical extensions that are linearly disjoint with the $n$-th cyclotomic extension of $K$. This result generalizes the description of Kummer Galois extensions as radical extensions of a field containing the $n$-th roots of the unity. The main tool is the construction of a product Hopf-Galois structure on the compositum of almost classically Galois extensions $L_1/K$, $L_2/K$ such that $L_1\cap M_2=L_2\cap M_1=K$, where $M_i$ is a field such that $L_iM_i=\widetilde{L}_i$, the normal closure of $L_i/K$. When $L/K$ is an extension of number or $p$-adic fields, we shall derive criteria on the freeness of the ring of integers $\mathcal{O}_L$  over its associated order in an almost classically Galois structure on $L/K$.
\end{abstract}

\MSC{12F10, 11Z05, 16T05, 11R04, 11R18, 11S15}


\section{Introduction}

Let $n$ be a positive integer and let $K$ be a field whose characteristic is coprime to $n$. If $K$ contains the $n$-th roots of the unity, a Galois field extension $L/K$ is said to be Kummer if the Galois group of $L/K$ is abelian of exponent $n$. These extensions owe their name to Ernst Kummer, who proved that these are the extensions of $K$ arising from the adjunction of $n$-th roots of elements in $K$. Namely, Kummer extensions $L/K$ are of the form $L=K(\alpha_1,\dots,\alpha_k)$ where $\alpha_i^n\in K$ for all $1\leq i\leq n$. The elements $\alpha_i$ are usually referred to as Kummer generators. Among the radical extensions, the ones obtained by adjoining to $K$ a single $n$-th root of an element in $K$ will be called simple radical. As Kummer extensions, they are just the cyclic extensions of $K$.

A typical limitation to classical Kummer theory is the requirement that the ground field contains a primitive $n$-th root of unity. For instance, this implies that the only Kummer extensions of $\mathbb{Q}$ are quadratic. Some authors have considered generalizations of this theory to further classes of extensions \cite{takahashi1968,greither1991,komatsu2009,perucca2020}.

In this paper we will develop a Kummer theory for the extensions $L=K(\alpha_1,\dots,\alpha_k)$ with $\alpha_i^n\in K$ such that $L\cap K(\zeta_n)=K$, where $\zeta_n$ is a primitive $n$-th root of the unity. Under this assumption, the normal closure $\widetilde{L}$ of $L/K$ is a Kummer extension of $K(\zeta_n)$ that shares many properties with the extension $L/K$, even though the latter is not Galois. A suitable setting to visualize $L/K$ as a generalized Kummer extension is the one provided by the theory of Hopf-Galois extensions.

The beginning of Hopf-Galois theory is the notion of Hopf-Galois structure on a finite extension $L/K$. This is a pair formed by a $K$-Hopf algebra $H$ and a $K$-linear action $\cdot\colon H\otimes_KL\longrightarrow L$ such that $L$ is an $H$-module algebra and the canonical map $j\colon L\otimes_KH\longrightarrow\mathrm{End}_K(L)$ is a $K$-linear isomorphism. A Hopf-Galois extension is a finite extension $L/K$ that admits some Hopf-Galois structure $(H,\cdot)$. We also say that $L/K$ is $H$-Galois. Under this definition, every Galois extension is Hopf-Galois but the converse does not hold in general. The notion of Hopf-Galois structure was introduced for the first time by Chase and Sweedler \cite{chasesweedler}, and it has been proved as a useful tool to study problems from classical Galois theory in a more general setting. The main research lines and outcomes of this theory so far have been summarized in the books \cite{childs,childsetal}.

Radical extensions as described above are particular instances of almost classically Galois extensions, a class of Hopf-Galois extensions that are naturally associated to Galois extensions. Concretely, a separable extension of fields $L/K$ is said to be almost classically Galois if for the normal closure $\widetilde{L}$ of $L/K$, the group $G'\coloneqq\mathrm{Gal}(\widetilde{L}/L)$ has some normal complement $J$ in $G\coloneqq\mathrm{Gal}(\widetilde{L}/K)$. The fixed subfield $M=\widetilde{L}^J$ will be referred to as the complement of $L/K$. Associated to $M$ we can construct a Hopf-Galois structure $H$ on $L/K$, which will be referred to as the almost classically Galois structure corresponding to $M$. In particular, an almost classically Galois extension is Hopf-Galois. Note that the extension $\widetilde{L}/M$ is Galois with group $J$, and it has the same degree as $L/K$. We will say that $L/K$ is almost abelian (resp. almost cyclic, resp. almost Kummer) if $\widetilde{L}/M$ is abelian (resp. cyclic, resp. Kummer). Accordingly, the exponent of $L/K$ is defined as the exponent of the group $J$.

\[
\xymatrix@=0.5cm{ & \widetilde{L} \ar@{-}[dd]_G & \\
L \ar@{-}[ur]^{G'} \ar@{-}[dr] & & \ar@{-}[ul]_J \ar@{-}[dl] M \\
& K &
}
\]

Our generalization of Kummer extension arises from the remark that the Kummer generators of a Kummer extension are eigenvectors of all the elements in its Galois group, where these are regarded as $K$-linear endomorphisms of $L$. Note that any Kummer extension admits a finite generating set of Galois eigenvectors, namely its Kummer generators. Here by generating set for an extension $L/K$ we mean a set $S\subseteq L$ such that $L=K(S)$. What is more, it turns out the finite Galois extensions of $K$ admitting a finite generating set of Galois eigenvectors are just the Kummer extensions of $K$. Now, we introduce the generalized Kummer extensions as follows: an $H$-Galois extension of $L/K$ is $H$-Kummer if it admits a finite generating set of elements in $L$ that are eigenvectors of all the elements of $H$, where these are regarded as endomorphisms by the corresponding action of $H$ on $L$ (see Definition \ref{defihkummer}).

We will be able to translate the correspondence between Kummer extensions and radical extensions from the classical case to this more general situation. We will need the following notion: two almost classically Galois extensions $L_1$ and $L_2$ of a field $K$ with complements $M_1$ and $M_2$ are said to be strongly disjoint if $L_1\cap M_2=L_2\cap M_1=K$. An extension that can be written as a compositum of pairwise strongly disjoint almost classically Galois extensions will be called strongly decomposable.

\begin{teo}\label{maintheorem1} Let $n\in\mathbb{Z}_{>0}$, let $K$ be a field with characteristic coprime to $n$ and let $M=K(\zeta_n)$. Let $L/K$ be a strongly decomposable extension of the form $L=K(\alpha_1,\dots,\alpha_k)$ with $\alpha_1,\dots,\alpha_k\in L$ and such that $K(\alpha_i)$, $K(\alpha_j)$ are strongly disjoint almost classically Galois extensions whenever $i\neq j$. The following statements are equivalent:
\begin{enumerate}
    \item\label{mainthm11} $L\cap M=K$, $\alpha_i^n\in K$ for every $1\leq i\leq k$ and $n$ is minimal for this property.
    \item\label{mainthm12} $L/K$ is an almost Kummer extension of exponent $n$ with complement $M$ and $\alpha_1,\dots,\alpha_k$ are $H$-eigenvectors of $L$, where $H$ is the almost classically Galois structure on $L/K$ corresponding to $M$.
\end{enumerate}
In particular, within the strongly decomposable extensions of $K$, the $n$-radical extensions that are linearly disjoint with $M$ are the almost Kummer extensions of exponent $n$ with complement $M$ that are $H$-Kummer.
\end{teo}

In order to prove Theorem \ref{maintheorem1}, we shall show that the tensor product of Hopf-Galois structures on two strongly disjoint almost classically Galois extensions $L_1/K$ and $L_2/K$ is a Hopf-Galois structure on their compositum $L/K$, which we call the product Hopf-Galois structure of those. Namely, if $H_i$ is a Hopf-Galois structure on $L_i/K$ for $i\in\{1,2\}$, one can construct a Hopf-Galois structure $H$ on $L_1L_2/K$, in such a way that $H\cong H_1\otimes_KH_2$ as $K$ algebras. This notion can be seen as an analogue for almost classically Galois extensions of the induced Hopf-Galois structures introduced by Crespo, Rio and Vela \cite{cresporiovela}. However, it is not a generalization nor a particular case of those: they apply to different pairs of extensions, except in the case of Galois extensions, for which both notions coincide.

What the last sentence of Theorem \ref{maintheorem1} means is that, within the strongly decomposable extensions of $K$, there is a bijective correspondence between radical extensions $L/K$ with $L\cap K(\zeta_n)=K$ and almost Kummer extensions with complement $K(\zeta_n)$ that are $H$-Kummer. Among these, simple radical extensions $L/K$ with $L\cap K(\zeta_n)=K$ correspond to almost cyclic extensions with complement $K(\zeta_n)$ that are $H$-cyclic. This can be seen as a direct generalization of the well known correspondence in classical Kummer theory.

We will investigate the module structure of $H$-Kummer extensions with the point of view of Hopf-Galois theory. For an $H$-Galois field extension $L/K$, we assume that $K$ is the field of fractions of some Dedekind domain $\mathcal{O}_K$ and we write $\mathcal{O}_L$ for the integral closure of $\mathcal{O}_K$ in $L$ (for instance, this holds when $L/K$ is an extension of number or $p$-adic fields). In short, we will say that $L/K$ is an extension \textit{with associated rings of integers}. Under this situation, the associated order in $H$ is defined as the set $\mathfrak{A}_H$ of elements of $H$ whose action on $L$ leave the ring of integers $\mathcal{O}_L$ invariant. The associated order is an $\mathcal{O}_K$-order in $H$ and $\mathcal{O}_L$ is naturally endowed with $\mathfrak{A}_H$-module structure. Under the assumption that $\mathcal{O}_L$ is $\mathcal{O}_K$-free, we have that if $\mathcal{O}_L$ is $\mathfrak{A}_H$-free, then it has a single generator. If $\mathfrak{A}$ is an $\mathcal{O}_K$-order in $H$ such that $\mathcal{O}_L$ is $\mathfrak{A}$-free, then $\mathfrak{A}=\mathfrak{A}_H$. However, in general $\mathcal{O}_L$ is not $\mathfrak{A}_H$-free, and it is interesting to find criteria in order to determine the $\mathfrak{A}_H$-freeness of such an extension. The study of these questions is often called Hopf-Galois module theory.

This is a natural generalization of the situation of a Galois tamely ramified extension $L/K$ with group $G$, for which, in the $p$-adic case, having a normal integral basis is equivalent to $\mathcal{O}_L$ being $\mathcal{O}_K[G]$-free, and for wildly ramified extensions we consider instead the associated order in $K[G]$. This line has been explored almost exclusively for tamely ramified extensions, both in the Galois case \cite{gomezayala,ichimura2004,delcorsorossi2010,delcorsorossi2013} and in the Hopf-Galois one \cite{truman2020}. Rio and the author \cite{gilrioinduced} introduced a method to study this kind of questions for $H$-Galois extensions $L/K$ from the knowledge of the action of $H$ on an $\mathcal{O}_K$-basis of $\mathcal{O}_L$, and it can be applied to wildly ramified extensions as well. In our situation, we will obtain the following:

\begin{teo}\label{maintheorem2} Let $L/K$ be an $H$-Kummer extension with associated rings of integers and assume that $L/K$ admits some basis of $H$-eigenvectors $B=\{\gamma_j\}_{j=1}^{n}$ which is also an $\mathcal{O}_K$-basis of $\mathcal{O}_L$. Let $W=\{w_i\}_{i=1}^n$ be a $K$-basis of $H$. Write $w_i\cdot\gamma_j=\lambda_{ij}\gamma_j$ with $\lambda_{ij}\in K$ for all $1\leq i,j\leq n$ and let $\Omega=(\omega_{ij})_{i,j=1}^n$ be the inverse of the matrix $\Lambda=(\lambda_{ji})_{i,j=1}^n$. Then:
\begin{itemize}
    \item[1.] An $\mathcal{O}_K$-basis of $\mathfrak{A}_H$ is given by the elements $v_i=\sum_{l=1}^n\omega_{li}w_l$, $1\leq i\leq n$. Moreover, they form a system of primitive pairwise orthogonal idempotents.
    \item[2.] $\mathcal{O}_L$ is $\mathfrak{A}_H$-free and a generator is any element $\beta=\sum_{j=1}^n\beta_j\gamma_j$ such that $\beta_j\in\mathcal{O}_K^*$ for every $1\leq j\leq n$.
\end{itemize}
\end{teo}

It is possible to obtain some criteria for the freeness for radical extensions of number or $p$-adic fields. We will follow this strategy: From Theorem \ref{maintheorem1} we know how to construct radical extensions $L/K$ that are $H$-Kummer, which have some $K$-basis of $H$-eigenvectors. Then we will add sufficient conditions to ensure the existence of an integral $K$-basis of $H$-eigenvectors, so that we can apply Theorem \ref{maintheorem2}. 

These extra conditions will depend on the nature of the fields involved. If we are working with extensions of number fields, the existence of integral generators that are $n$-th roots of elements in $K$ is enough. This is related with the monogeneity of $L/K$, which has been widely studied in literature. In the case of $p$-adic fields, we will restrict to simple radical extensions, and in that case we will need some uniformizer of $L$ that is an $n$-th root of some element in $K$. Using these considerations, we will prove that a Hopf-Galois prime degree extension with maximally ramified normal closure accomplishes the freeness property over the associated order.

This paper is organized as follows. Section \ref{sectprelim} includes the main background knowledge on Hopf-Galois theory required to follow the main ideas of this paper, namely the notion of Hopf-Galois structure, Hopf-Galois theory on separable extensions (specifically, Greither-Pareigis theory) and Hopf-Galois module theory. In Section \ref{sectprodhg} we will investigate the compositums of almost classically Galois extensions. We will introduce the notion of strong disjointness and define the product Hopf-Galois structure on a compositum of strongly disjoint almost classically Galois extensions. Section \ref{kummergalois} will be devoted to a review of basic results concerning Kummer Galois extensions and their cyclic subextensions, as well as a proof of the characterization of the Kummer condition in terms of the Galois action. In Section \ref{sect:kummerhopfgalois} we will introduce the notion of $H$-eigenvector and $H$-Kummer extension. The aim of Section \ref{sect:corresprad} will essentially be to prove Theorem \ref{maintheorem1}, and we will extract some consequences. Finally, in Section \ref{sect:modstr} we will consider the problem of the module structure of $\mathcal{O}_L$ described above, and we will prove Theorem \ref{maintheorem2}.

\section{Preliminaries}\label{sectprelim}

\subsection{Hopf-Galois structures and Greither-Pareigis theory}\label{secthgtheory}

Let $L/K$ be a finite extension of fields. A Hopf-Galois structure on $L/K$ is a pair $(H,\cdot)$ where $H$ is a finite-dimensional cocommutative $K$-Hopf algebra and $\cdot$ is a $K$-linear action of $H$ on $L$ such that:
\begin{itemize}
    \item The action $\cdot$ endows $L$ with $H$-module algebra structure, that is, the following conditions are satisfied for every $h\in H$: $$h\cdot1=\epsilon_H(h),$$ $$h\cdot(xx')=\sum_{(h)}(h_{(1)}\cdot x)(h_{(2)}\cdot x'),\quad x,x'\in L,$$ where $\epsilon_H\colon H\longrightarrow K$ is the counity of $H$ and the comultiplication $\Delta_H\colon H\longrightarrow H\otimes_K H$ of $H$ satisfies $\Delta_H(h)=\sum_{(h)}h_{(1)}\otimes h_{(2)}$.
    \item The canonical map $j\colon L\otimes_KH\longrightarrow\mathrm{End}_K(L)$ defined by $j(x\otimes h)(y)=x(h\cdot y)$ for every $y\in L$ is an isomorphism of $K$-vector spaces.
\end{itemize} A Hopf-Galois extension is a finite extension $L/K$ that admits some Hopf-Galois structure. For the sake of simplicity, we will denote a Hopf-Galois structure simply by $H$ (so that the action $\cdot$ is implicit in the context). We will usually say that $L/K$ is $H$-Galois.

If $L/K$ is a Galois extension with group $G$, then $K[G]$ together with its Galois action on $L$ (extended by $K$-linearity) is a Hopf-Galois structure on $L/K$, commonly referred to as the classical Galois structure. There are many Hopf-Galois extensions that are not Galois, for instance $\mathbb{Q}(\sqrt[3]{2})/\mathbb{Q}$.

A single Hopf-Galois extension may admit different Hopf-Galois structures. In the case that the extension $L/K$ is separable, Greither and Pareigis \cite[Theorem 2.1]{greitherpareigis} found a characterization of the Hopf-Galois structures in terms of group theory. Concretely, let $\widetilde{L}$ be the normal closure of $L/K$, $G=\mathrm{Gal}(\widetilde{L}/K)$, $G'=\mathrm{Gal}(\widetilde{L}/L)$ and $X=G/G'$. 

\begin{teo}[Greither-Pareigis theorem]\label{thm:greitherpareigis} The Hopf-Galois structures on $L/K$ are in bijective correspondence with the regular subgroups of $\mathrm{Perm}(X)$ normalized by $\lambda(G)$.
\end{teo}

In this statement, we say that a subgroup $N$ of $\mathrm{Perm}(X)$ is:
\begin{itemize}
    \item Regular, if its action on $X$ by evaluation is simply transitive.
    \item Normalized by $\lambda(G)$, if $\lambda(g)\eta\lambda(g)^{-1}\in N$ for every $g\in G$, where $\lambda\colon G\longrightarrow\mathrm{Perm}(X)$, defined by $\lambda(g)(hG')=ghG'$, is the left translation map of $L/K$.
\end{itemize} Moreover, there is an isomorphism of $\widetilde{L}$-algebras $\widetilde{L}\otimes_KH\longrightarrow\widetilde{L}[N]$, and by descent theory this gives $$H=\widetilde{L}[N]^G=\{h\in\widetilde{L}[N]\,|\,g(h)=h\hbox{ for all }g\in G\},$$ where $G$ acts by evaluation on $\widetilde{L}$ and by conjugation by $\lambda(G)$ on $N$. The action of $H$ on $L$ is as follows: an element $h=\sum_{i=1}^nh_i\eta_i\in H$ with $h_i\in\widetilde{L}$ and $\eta_i\in N$ acts on $\alpha\in L$ by \begin{equation}\label{hopfactiongp}
    h\cdot\alpha=\sum_{i=1}^nh_i\eta_i^{-1}(1_GG')(\alpha).
\end{equation}

In the case that $L/K$ is Galois, we have that $G'$ is trivial, so the left translation map becomes $\lambda\colon G\longrightarrow\mathrm{Perm}(G)$, the left regular representation of $G$ in $\mathrm{Perm}(G)$. Let $\rho\colon G\longrightarrow\mathrm{Perm}(G)$ be defined as $\rho(g)(g')=g'g^{-1}$. Then $\rho(G)$ and $\lambda(G)$ are regular subgroups both normalized by $\lambda(G)$, giving rise to Hopf-Galois structures on $L/K$. The one given by $\rho(G)$ is the classical Galois structure on $L/K$ (see \cite[(6.10)]{childs}).

\subsubsection{Induced Hopf-Galois structures}\label{sect:inducedhgstr}

The notion of induced Hopf-Galois structure was originally introduced by Crespo, Rio and Vela \cite{cresporiovela}. Let $L/K$ be a Galois extension with group of the form $G=J\rtimes G'$ with $J$ normal in $G$, and let $E=L^{G'}$. The induction theorem essentially states that under these hypothesis, we can construct a Hopf-Galois structure on $E/K$ from Hopf-Galois structures on $E/K$ and $L/E$ by carrying out the direct product of the corresponding permutation subgroups under the Greither-Pareigis correspondence (see \cite[Theorem 3]{cresporiovela}). We present here the reformulation of this notion by Rio and the author \cite[Section 5]{gilrioinduced}.

First of all, by the Greither-Pareigis theorem, the Hopf-Galois structures on $L/E$ are in one-to-one correspondence with the regular subgroups of $\mathrm{Perm}(G')$ normalized by the image of the left regular representation $\lambda'\colon G'\longrightarrow\mathrm{Perm}(G')$. On the other hand, in order to describe the Hopf-Galois structures on $E/K$, it can be proved that we can use the Greither-Pareigis theorem as if $L=\widetilde{E}$, even if it is not (in general, $\widetilde{E}\subseteq L$). Hence, the Hopf-Galois structures on $E/K$ correspond bijectively to the regular subgroups of $\mathrm{Perm}(G/G')$ normalized by $\overline{\lambda}(G)$, where $\overline{\lambda}\colon G\longrightarrow\mathrm{Perm}(G/G')$, $\overline{\lambda}(g)(hG')=ghG'$. Now, since $G=J\rtimes G'$, $J$ is a transversal of $G/G'$, and then $\mathrm{Perm}(G/G')\cong\mathrm{Perm}(J)$, which yields a map $\lambda_c\colon G\longrightarrow\mathrm{Perm}(J)$. On this way, the Hopf-Galois structures on $E/K$ correspond bijectively to the regular subgroups of $\mathrm{Perm}(J)$ normalized by $\lambda_c(G)$.

Finally, the Hopf-Galois structures on $L/K$ are in bijective correspondence with the regular subgroups of $\mathrm{Perm}(G)$ normalized by the image of the left regular representation $\lambda\colon G\longrightarrow\mathrm{Perm}(G)$. It can be checked that $\lambda=\iota\circ\chi$, where $\iota$ and $\chi$ are the group homomorphisms given by \begin{equation}\label{eq:declambda}\begin{array}{rcccr}
    \chi\colon & G & \longrightarrow & \mathrm{Perm}(J)\times\mathrm{Perm}(G') &\\
     & \sigma\tau & \longmapsto & (\lambda_c(\sigma\tau),\lambda'(\tau)), &\\
     \\
    \iota\colon & \mathrm{Perm}(J)\times\mathrm{Perm}(G') & \longrightarrow & \mathrm{Perm}(G)&\\
     & (\varphi,\psi) & \longmapsto & \sigma\tau\mapsto\varphi(\sigma)\psi(\tau).&
\end{array}\end{equation} Now, induced Hopf-Galois structures are introduced as follows:

\begin{pro}\label{pro:ngivesinduced} If $N_1\leq\mathrm{Perm}(J)$ gives $E/K$ a Hopf-Galois structure and $N_2\leq\mathrm{Perm}(G')$ gives $L/E$ a Hopf-Galois structure, then $N=\iota(N_1\times N_2)$ gives $L/K$ a Hopf-Galois structure, which is called induced.
\end{pro}

Actually, the Hopf-Galois structures on $L/E$ are in bijective correspondence with the ones of $F/K$, where $F=L^J$ (see \cite[Proposition 5.3]{gilrioinduced}), so an induced Hopf-Galois structure on $L/K$ can be built equivalently from Hopf-Galois structures on $E/K$ and $F/K$. This point of view is more convenient in order to study the underlying Hopf algebra and the underlying action on the induced Hopf-Galois structure.

\begin{pro}\label{pro:prodindhgstr} Let $H$ be an induced Hopf-Galois structure on $L/K$ from Hopf-Galois structures $H_1$ on $E/K$ and $H_2$ on $F/K$. Then:
\begin{itemize}
    \item[1.]\cite[Proposition 5.5]{gilrioinduced} $H\cong H_1\otimes_KH_2$ as $K$-algebras.
    \item[2.]\cite[Proposition 5.8]{gilrioinduced} If $h_i\in H_i$ and $\alpha_i\in L_i$ for $i\in\{1,2\}$, then $(h_1h_2)\cdot(\alpha_1\alpha_2)=(h_1\cdot \alpha_1)(h_2\cdot \alpha_2)$.
\end{itemize}
\end{pro}

\subsection{Almost classically Galois extensions}\label{sect:almostclassic}

Almost classically Galois extensions were introduced also by Greither and Pareigis \cite[Section 4]{greitherpareigis}.

\begin{teo}\label{charactalmostclassical}\cite[Proposition 4.1]{greitherpareigis} Let $L/K$ be a separable extension and let $\widetilde{L}$ be its normal closure. Let $G=\mathrm{Gal}(L/K)$, $G'=\mathrm{Gal}(\widetilde{L}/L)$ and $X=G/G'$. Then, the following statements are equivalent:
	\begin{enumerate}
		\item\label{almostclassiccond1} There is some Galois extension $M/K$ such that $L\otimes_K M$ is a field that contains $\widetilde{L}$.
		\item\label{almostclassiccond2} There is some Galois extension $M/K$ such that $L\otimes_K M=\widetilde{L}$.
		\item\label{almostclassiccond3} There is some normal complement $J$ of $G'$ in $G$.
		\item\label{almostclassiccond4} There is a regular subgroup $N$ of $\mathrm{Perm}(X)$ normalized by $\lambda(G)$ such that $N\subset\lambda(G)$.
	\end{enumerate}
\end{teo}

\begin{defi} Let $L/K$ be a separable extension. We say that $L/K$ is \textbf{almost classically Galois} if it satisfies the equivalent conditions at Theorem \ref{charactalmostclassical}. An extension $M/K$ as in \ref{almostclassiccond2} will be called a Galois complement (or simply a complement) for $L/K$. A normal complement for $L/K$ will be a subgroup $J$ of $G$ as in \ref{almostclassiccond3}.
\end{defi}

It follows from Theorem \ref{charactalmostclassical} \ref{almostclassiccond4} and the Greither-Pareigis theorem \ref{thm:greitherpareigis} that every almost classically Galois extension is Hopf-Galois. Also, every Galois extension is almost classically Galois: in that case, we have that $G'$ is trivial, and hence it has the full Galois group as a normal complement.

\begin{rmk}\normalfont\label{rmk:almostclass} The relationship between the objects involved in the statements of Theorem \ref{charactalmostclassical} is as follows. If a subgroup $J$ of $G$ is a normal complement of $G'$, then $N=\lambda(J)$ is as in \ref{almostclassiccond4}. However, there might be a subgroup $J$ for which $N=\lambda(J)$ is as in \ref{almostclassiccond4} but $J$ is not normal, and then it is not a normal complement of $G'$ (in that case, we only know by the theorem that $G'$ has some normal complement). On the other hand, a subgroup $J$ is a normal complement of $G'$ if and only if for $M=L^J$, $M/K$ is a Galois complement for $L/K$.
\end{rmk}

Let $L/K$ be an almost classically Galois extension with normal complement $J$ and Galois complement $M$. Then $\widetilde{L}/M$ is a Galois extension with group $J$, and it has the same degree as the extension $L/K$. We can define properties of Galois extensions in this setting by means of the extension $\widetilde{L}/M$.

\begin{defi} Let $L/K$ be an almost classically Galois extension.
\begin{itemize}
    \item[1.] We say that $L/K$ is almost cyclic (resp. almost abelian) if it admits some normal complement $J$ which is cyclic (resp. abelian).
    \item[2.] We say that $L/K$ has exponent $n\in\mathbb{Z}_{>0}$ if it admits some normal complement $J$ which has exponent $n$.
    \item[3.] We say that $L/K$ is almost Kummer with respect to $n$ if it is almost abelian with exponent dividing $n$.
\end{itemize}
\end{defi}

The notion of almost cyclic extension had already been introduced in \cite[Definition 3.1]{byott2007}.

It is also possible to define what we understand by an almost classical Galois structure. For a group $(N,\cdot)$, we write $N^{\mathrm{opp}}$ for its opposite group, i.e. the centraliser of $N$ within $\mathrm{Perm}(X)$. If $N$ is abelian, we simply have that $N^{\mathrm{opp}}=N$. Now, assume that $N$ is a regular subgroup of $\mathrm{Perm}(X)$. Then it is checked that $N^{\mathrm{opp}}$ is also a regular subgroup of $\mathrm{Perm}(X)$ (see \cite[Proof of Theorem 2.5 (b)]{greitherpareigis}). Then, the following definition makes sense.

\begin{defi}\label{defialmoststr} Let $L/K$ be an almost classically Galois extension. Let $H$ be a Hopf-Galois structure on $L/K$ and let $N$ be the corresponding regular subgroup of $\mathrm{Perm}(X)$ normalized by $\lambda(G)$. We say that $H$ is an \textbf{almost classically Galois structure} if $N^{\mathrm{opp}}\subset\lambda(G)$. When a permutation subgroup giving an almost classically Galois structure $H$ is of the form $\lambda(J)^{\mathrm{opp}}$ for a normal complement $J$ of $L/K$, we will say that $H$ corresponds to $M\coloneqq\widetilde{L}^J$.
\end{defi}

As highlighted in Remark \ref{rmk:almostclass}, not all almost classically Galois structures are as in the second part of Definition \ref{defialmoststr}. However, they are especially well behaved. If $H$ is the Hopf algebra in an almost classically Galois structure on $L/K$, we have that $H=(\widetilde{L}[N]^J)^{G/J}$ (see \cite[Proof of Theorem 3.1]{KKTU19}). Now, if $N=\lambda(J)^{\mathrm{opp}}$ for a normal complement $J$ of $L/K$, the action of $J$ on $N$ by conjugation by $\lambda(J)$ is trivial. Then, we obtain that $H=M[N]^{G'}$. Identifying $\sigma$ with $\lambda(\sigma)$ (which defines an isomorphism of groups $J\cong\lambda(J)$), we conclude the following.

\begin{pro}\label{almostclassicstr} Let $L/K$ be an almost classically Galois extension of fields, and let $J$ be a normal complement as in Theorem \ref{charactalmostclassical}. Let $H$ be the almost classically Galois structure on $L/K$ corresponding to $J$. Then, $H=M[J^{\mathrm{opp}}]^{G'}$. If in addition $J$ is abelian, then $H=M[J]^{G'}$.
\end{pro}

Let $L/K$ be a Galois extension which we regard as an almost classically Galois extension. As aforesaid, the normal complement of $L/K$ for this case is $J=G$, and accordingly the Galois complement is $M=K$. Then, the almost classically Galois structure on $L/K$ corresponding to $M$ is the one given by the permutation subgroup $N=\lambda(G)^{\mathrm{opp}}$. Now, it is easily checked that $\lambda(G)$ normalizes $\rho(G)$, and since both are regular subgroups, we obtain that $N=\rho(G)$. In other words, the almost classically Galois structure on a Galois extension corresponding to its Galois complement is just its classical Galois structure.

\subsection{Linearly disjoint extensions}

Let $K$ be a field and let $L_1$ and $L_2$ be field extensions of $K$ contained in the algebraic closure $\overline{K}$ of $K$. Then the morphism $$L_1\otimes_KL_2\longrightarrow L_1L_2$$ defined by $x_1\otimes x_2\mapsto x_1x_2$ and extended by $K$-linearity is always surjective. We say that $L_1$ and $L_2$ are $K$-linearly disjoint, or that $L_1/K$ and $L_2/K$ are linearly disjoint, if the map above is bijective, or equivalently, if $L_1\otimes_KL_2$ is a field. It is easy to see that if $L_1$ and $L_2$ are $K$-linearly disjoint, then $L_1\cap L_2=K$. The converse in general does not hold.

\begin{example}\normalfont Let $K=\mathbb{Q}$, $L_1=\mathbb{Q}(\alpha)$ and $L_2=\mathbb{Q}(\zeta_3\alpha)$, where $\alpha=\sqrt[3]{2}$ is the only real number with $\alpha^3=2$ and $\zeta_3$ is a primitive third root of the unity. Then $L_1\cap L_2=\mathbb{Q}$ because $L_1\subset\mathbb{R}$ and $L_2\cap\mathbb{R}=\mathbb{Q}$. In addition, we have that $L_1L_2=\mathbb{Q}(\alpha,\zeta_3\alpha)=\mathbb{Q}(\alpha,\zeta_3)$, which is easily checked to have degree $6$ of $\mathbb{Q}$. Since $\mathrm{dim}_{\mathbb{Q}}(L_1\otimes_{\mathbb{Q}}L_2)=9$, we have that $L_1L_2$ and $L_1\otimes_{\mathbb{Q}}L_2$ are not isomorphic as $\mathbb{Q}$-algebras, and then $L_1$ and $L_2$ are not $\mathbb{Q}$-linearly disjoint.
\end{example}

However, the converse is true under fairly general restrictions on $L_1/K$ and $L_2/K$.

\begin{pro}\label{pro:lindisjcharact}\cite[Theorem 5.5]{cohn1991} Let $L_1/K$ and $L_2/K$ be finite extensions of fields such that one of them is normal and one (possibly the same) is separable. Then $L_1/K$ and $L_2/K$ are linearly disjoint if and only if $L_1\cap L_2=K$.
\end{pro}

In the case that neither of the extensions is normal, proving the linear disjointness may be a tricky problem. The following result gives a sufficient condition for specific families of extensions (see \cite[Theorem]{mordell}).

\begin{pro}\label{mordell1950th} Let $n_1,\dots,n_k\in\mathbb{Z}_{>0}$ and let $K$ be a number field. Let $L$ be an extension of $K$ generated by elements $\alpha_1,\dots,\alpha_k\in\overline{K}$ such that $\alpha_i^{n_i}\in K$ for all $1\leq i\leq k$ and $\alpha_i^{k_i}\notin K$ for all $1\leq k_i<n_i$. Assume that $K$ is totally real or $\zeta_{n_i}\in K$ for all $1\leq i\leq k$. Then the fields $K(\alpha_1),\dots,K(\alpha_k)$ are pairwise $K$-linearly disjoint.
\end{pro}

As for the relation between linear disjointness and almost classically Galois extensions, the following result is immediately deduced from Theorem \ref{charactalmostclassical} and Proposition \ref{pro:lindisjcharact}:

\begin{pro}\label{pro:lindisjalmostclassic} Let $L/K$ and $M/K$ be separable extensions with $M/K$ Galois. Then $L/K$ is almost classically Galois with Galois complement $M$ if and only if $L\cap M=K$ and $\widetilde{L}=LM$.
\end{pro}

\begin{rmk}\normalfont Two almost classically Galois extensions $L_1/K$ and $L_2/K$ such that $\mathrm{gcd}([L_1:K],[L_2:K])=1$ are linearly disjoint, but the converse does not hold in general. For instance, the fields $\mathbb{Q}(\sqrt[4]{2})$ and $\mathbb{Q}(i)$ are $\mathbb{Q}$-linearly disjoint almost classically Galois extensions of $\mathbb{Q}$.
\end{rmk}

In the case of extensions of number or $p$-adic fields, it is possible to consider a stronger notion than linear disjointness.

\begin{pro} Two extensions of number or $p$-adic fields with the same ground field are said to be arithmetically disjoint if they are linearly disjoint and have coprime discriminants.
\end{pro}

It is known that if $L_1/K$ and $L_2/K$ are arithmetically disjoint, then $\mathcal{O}_{L_1L_2}=\mathcal{O}_{L_1}\otimes_{\mathcal{O}_K}\mathcal{O}_{L_2}$ (see \cite[(2.13)]{frohlichtaylor}).

\subsection{Hopf-Galois module theory}\label{sect:hgmodtheory}

Let $K$ be the fraction field of a Dedekind domain $\mathcal{O}_K$, let $L$ be a degree $n$ Hopf-Galois extension of $K$ and let $\mathcal{O}_L$ be the integral closure of $\mathcal{O}_K$ in $L$. Let $(H,\cdot)$ be a Hopf-Galois structure on $L/K$. The associated order of $\mathcal{O}_L$ in $H$ is defined as $$\mathfrak{A}_H=\{h\in H\,|\,h\cdot\alpha\in\mathcal{O}_L\hbox{ for every }\alpha\in\mathcal{O}_L\}.$$ This is an $\mathcal{O}_K$-order in $H$ and, under the assumption that $\mathcal{O}_K$ is a PID, it is $\mathcal{O}_K$-free of rank $n$. Moreover, it is known that if $\mathfrak{A}$ is an $\mathcal{O}_K$-order in $H$ such that $\mathcal{O}_L$ is $\mathfrak{A}$-free, then $\mathfrak{A}=\mathfrak{A}_H$. If $L/K$ is Galois, we will write $\mathfrak{A}_{L/K}$ for the associated order in the classical Galois structure on $L/K$.

When $\mathcal{O}_K$ is a PID, Rio and the author \cite[Section 3]{gilrioinduced} established a constructive method to determine an $\mathcal{O}_K$-basis of $\mathfrak{A}_H$. We summarize here the main lines.

Let us fix $K$-bases $W=\{w_i\}_{i=1}^n$ and $B=\{\gamma_j\}_{j=1}^n$ of $H$ and $L$ respectively. Then, we can define a $K$-basis $\Phi=\{\varphi_i\}_{i=1}^{n^2}$ of $\mathrm{End}_{K}(L)$ as follows: For every $1\leq i\leq n^2$, there are $1\leq k,j\leq n$ such that $i=k+(j-1)n$. Let $\varphi_i$ be the map that sends $\gamma_j$ to $\gamma_k$ and the other $\gamma_l$ to $0$. The {matrix of the action} of $H$ on $L$ with respect to the bases $W$ and $B$ is the matrix $M(H_W,L_B)$ of the linear map $\rho_H\colon H\longrightarrow\mathrm{End}_{K}(L)$ arising from the choice of the basis $W$ in $H$ and the basis $\Phi$ in $\mathrm{End}_K(L)$. Equivalently, $$M(H,L)=\begin{pmatrix}
M_1(H,L) \\ \hline
\cdots \\ \hline
M_{n}(H,L) \end{pmatrix}\in\mathcal{M}_{n^2\times n}(K),$$ where\begin{equation}\label{blocksmatrix}
    M_j(H,L)\coloneqq\begin{pmatrix}
|& | &\dots  &|  \\
(w_1\cdot\gamma_j)_B&(w_2\cdot\gamma_j)_B&\dots &(w_n\cdot\gamma_j)_B \\
|& |&\dots & |\\
\end{pmatrix}\in \mathcal{M}_n(K)
\end{equation} for every $1\leq j\leq n$.

Now, there is a matrix $D\in\mathcal{M}_n(K)$ and a unimodular matrix $U\in\mathrm{GL}_{n^2}(\mathcal{O}_L)$ with the property that $$UM(H,L)=\begin{pmatrix}D \\ \hline \\[-2ex] O\end{pmatrix},$$ where $O$ is the zero matrix of $\mathcal{M}_{(m-n)\times n}(K)$ (see \cite[Theorem 2.3]{gilrioquartic}). We will refer to such a matrix $D$ as a \textbf{reduced matrix}. Note that the injectivity of $\rho_H$ implies that $M(H,L)$ has rank $n$, and therefore any reduced matrix of $M(H,L)$ is invertible. Then, an $\mathcal{O}_K$-basis of $\mathfrak{A}_H$ is determined as follows.

\begin{pro}\label{pro:basisassocorder} Let $W$ be a $K$-basis of $H$ and let $B$ be a $K$-integral basis of $L$ (i.e, an $\mathcal{O}_K$-basis of $\mathcal{O}_L$). Let $D$ be a reduced matrix for $M(H,L)$. Then, the elements of $H$ whose coordinates with respect to $W$ are the columns of the matrix $D^{-1}$ form an $\mathcal{O}_K$-basis of $\mathfrak{A}_H$.
\end{pro}

A proof can be found in \cite[Theorem 3.5]{gilrioinduced}. In that reference the result is stated for a concrete reduced matrix, the Hermite normal form of $M(H,L)$ (or more accurately, of the matrix obtained from $M(H,L)$ by dropping out the denominators of its entries, which has coefficients in $\mathcal{O}_K$), but this does not make any difference, since any two reduced matrices differ by multiplication of an invertible matrix in $\mathcal{M}_n(\mathcal{O}_K)$.

Let $L/K$ be an extension with Galois group of the form $G=J\rtimes G'$. When $L^{G'}/K$ and $L^J/K$ are arithmetically disjoint, it is possible to establish a relation between the associated order in an induced Hopf-Galois structure on $L/K$ and the ones in the inducing Hopf-Galois structures, as well as the freeness of the rings of integers over the corresponding associated orders.

\begin{pro}\label{pro:indhgmstr} Let $K$ be the fraction field of a PID $\mathcal{O}_K$, let $L$ be an extension of $K$ with Galois group of the form $G=J\rtimes G'$ and let $E=L^{G'}/K$ and $F=L^J/K$. Let $H$ be an induced Hopf-Galois structure on $L/K$ from Hopf-Galois structures $H_1$ on $E/K$ and $H_2$ on $F/K$. Assume that $E/K$ and $F/K$ are arithmetically disjoint. Then:
\begin{enumerate}
    \item[1.]\cite[Theorem 5.11]{gilrioinduced} $\mathfrak{A}_H=\mathfrak{A}_{H_1}\otimes_{\mathcal{O}_K}\mathfrak{A}_{H_2}$.
    \item[2.]\cite[Theorem 5.16]{gilrioinduced} If $\mathcal{O}_{E}$ is $\mathfrak{A}_{H_1}$-free and $\mathcal{O}_F$ is $\mathfrak{A}_{H_2}$-free, then $\mathcal{O}_L$ is $\mathfrak{A}_H$-free.
\end{enumerate}
\end{pro}

\section{Products of Hopf-Galois structures on almost classically Galois extensions}\label{sectprodhg}

In this section we consider the compositum $L$ of almost classically Galois extensions $L_1/K$, $L_2/K$, which is isomorphic to their tensor product $L_1\otimes_KL_2$ when $L_1$ and $L_2$ are $K$-linearly disjoint. We are interested in finding conditions for $L_1$ and $L_2$ in order to assure that $L/K$ is almost classically Galois. This phenomenon is already known for Galois extensions by means of the following result, whose proof is straightforward.

\begin{pro}\label{pro:prodgalois} Let $E/K$ and $F/K$ be Galois extensions of fields with Galois groups $G_E$ and $G_F$, respectively. Write $L=EF$ for the compositum of $E$ and $F$. Then:
\begin{itemize}
    \item[1.] $L/K$ is Galois.
    \item[2.] The Galois group $G$ of $L/K$ is such that the map $f\colon G\longrightarrow G_E\times G_F$ defined by $f(\sigma)=(\sigma|_E,\sigma|_F)$ is injective.
    \item[3.] The map $f$ is bijective if and only if $E/K$ and $F/K$ are linearly disjoint.
\end{itemize}
\end{pro}

The second statement means that the Galois group $G$ of $L/K$ can be embedded in the direct product $G_E\times G_F$ by means of the map $f$. The action of $G$ on $L$ can be described as follows: Given $\sigma\in G$, for every $\alpha_E\in E$ and $\alpha_F\in F$ we have $$\sigma(\alpha_E\alpha_F)=\sigma(\alpha_E)\sigma(\alpha_F)=\sigma|_E(\alpha_E)\sigma|_F(\alpha_F),$$ with $f(\sigma)=(\sigma|_E,\sigma|_F)$. %We can identify $G$ with its image by $f$ in $G_1\times G_2$ by writing $\sigma_E\sigma_F\equiv(\sigma_E,\sigma_F)$. Then, for each $\sigma\in G$, there are unique $\sigma_E\in G_E$ and $\sigma_F\in G_F$ such that $\sigma=\sigma_E\sigma_F$. %Moreover, when $E$ and $F$ are $K$-linearly disjoint, the action of $G$ on $L$ is the product of the actions of $G_E$ and $G_F$, meaning that for every $\alpha_E\in E$ and $\alpha_F\in F$ we have $$(\sigma_E\sigma_F)(\alpha_E\alpha_F)=\sigma_E(\alpha_E)\sigma_F(\alpha_F).$$ This is because $\sigma_E=\sigma\mid_E$ and $\sigma_F=\sigma\mid_F$.

\subsection{The compositum of almost classically Galois extensions}

Let $L_1/K$ and $L_2/K$ be two almost classically Galois extensions with Galois complements $M_1$ and $M_2$ respectively. Assume that $L_1$ and $L_2$ are $K$-linearly disjoint; in particular $L_1\cap L_2=K$. We introduce the following terminology:

\begin{defi} Let $L_1/K$ and $L_2/K$ be almost classically Galois extensions with Galois complements $M_1$ and $M_2$ respectively. Assume that $L_1$ and $L_2$ are $K$-linearly disjoint. We say that $L_1/K$ and $L_2/K$ are \textbf{strongly disjoint} if $L_1\cap M_2=L_2\cap M_1=K$. Any extension that can be written as the compositum of strongly disjoint extensions will be called \textbf{strongly decomposable}.
\end{defi}

In this definition we are not requiring that $M_1\cap M_2=K$. In fact, it is quite easy to find strongly disjoint extensions $L_1,L_2/K$ with complements $M_1,M_2$ such that $M_1\cap M_2\neq K$.

\begin{rmk}\normalfont Note that the Galois complement $M$ of an almost classically Galois extension $L/K$ depends on its normal complement $J$, which in general is not unique. Hence, the notion of strong disjointness for almost classically Galois extensions $L_1/K$ and $L_2/K$ depends on the choice of the normal complements $J_1$ and $J_2$. Throughout this paper such a choice will always be implicit.
\end{rmk}

Write $L=L_1L_2$ for the compositum of $L_1$ and $L_2$. Recall that we want to show that $L/K$ is an almost classically Galois extension and build a Hopf-Galois structure on $L/K$ from Hopf-Galois structures on $L_1/K$ and $L_2/K$.

For $i\in\{1,2\}$, we introduce the following notation:
\begin{itemize}
    \item $\widetilde{L_i}$ is the normal closure of $L_i/K$.
    \item $G_i\coloneqq\mathrm{Gal}(\widetilde{L_i}/K)$, $G_i'\coloneqq\mathrm{Gal}(\widetilde{L_i}/L_i)$ and $J_i\coloneqq\mathrm{Gal}(\widetilde{L_i}/M_i)$.
    \item $\lambda_i\colon G_i\longrightarrow\mathrm{Perm}(G_i/G_i')$ is the left translation map for $L_i/K$.
    \item $N_i$ is a regular subgroup of $\mathrm{Perm}(G_i/G_i')$ normalized by $\lambda_i(G_i)$ (therefore, giving a Hopf-Galois structure on $L_i/K$).
\end{itemize}

Note that the normal closure of $L/K$ is $\widetilde{L}=\widetilde{L_1}\widetilde{L_2}$. Indeed, if $N/K$ is a normal extension such that $L\subseteq N$, then $L_1,L_2\subseteq N$, and the normality of $N$ gives that $\widetilde{L_1},\widetilde{L_2}\subseteq N$, so $\widetilde{L_1}\widetilde{L_2}\subseteq N$.

\begin{pro}\label{pro:compalmostclassical} Let $L_1/K$ and $L_2/K$ be strongly disjoint almost classically Galois extensions with complements $M_1$, $M_2$, and call $L=L_1L_2$. Then $L/K$ is an almost classically Galois extension with complement $M=M_1M_2$. Consequently, any strongly decomposable extension is almost classically Galois.
\end{pro}
\begin{proof}
Since $M_1/K$ and $M_2/K$ are Galois, by Proposition \ref{pro:prodgalois}, $M/K$ is also Galois. Moreover, $$\widetilde{L}=\widetilde{L_1}\widetilde{L_2}=(L_1M_1)(L_2M_2)=(L_1L_2)(M_1M_2)=LM.$$ Let us prove that $LM\cong L\otimes_KM$. We consider the formal $K$-algebra $L\otimes_KM$. We have that $$L\otimes_KM=(L_1\otimes_KM_1)(L_1\otimes_KM_2)(L_2\otimes_KM_1)(L_2\otimes_KM_2).$$ At the right side member of this equality, the factors $L_i\otimes_KM_j$ are regarded as $K$-subalgebras of $L\otimes_KM$, and their product is the $K$-algebra generated by all possible products (within $L\otimes_KM$) of their elements. Now, by definition of Galois complement, $\widetilde{L_i}\cong L_i\otimes_KM_i$ for $i\in\{1,2\}$, so that $(L_1\otimes_KM_1)(L_2\otimes_KM_2)\cong\widetilde{L}$ as $K$-algebras. Furthermore, the strong disjointness together with the fact that $L_1,L_2,M_1,M_2\subseteq\widetilde{L}$ allows us to apply Proposition \ref{pro:lindisjcharact}, obtaining that $L_i\otimes_KM_j\cong L_iM_j$ for every $1\leq i,j\leq 2$. Then $L_1\otimes_KM_2$ and $L_2\otimes_KM_1$ can be embedded in $\widetilde{L}$, proving that $L\otimes_KM\cong\widetilde{L}$.
\end{proof}

Let us call $G\coloneqq\mathrm{Gal}(\widetilde{L}/K)$, $J\coloneqq\mathrm{Gal}(\widetilde{L}/M)$ and $G'\coloneqq\mathrm{Gal}(\widetilde{L}/L)$. Applying Proposition \ref{pro:prodgalois} to the extensions $\widetilde{L_1}/K$, $\widetilde{L_2}/K$, we obtain that there is a monomorphism $G\hookrightarrow G_1\times G_2$. %Recall that we can identify $G$ with its image in $G_1\times G_2$, so that for each $g\in G$ there are unique $g_1\in G_1$, $g_2\in G_2$ such that $g=g_1g_2$. 

\begin{lema}\label{lema:embedding} Under the embedding $G\hookrightarrow G_1\times G_2$, we have that:
\begin{enumerate}
    \item\label{lemaemb1} $G'$ is embedded in $G_1'\times G_2'$.
    \item\label{lemaemb2} $J$ is isomorphic to $J_1\times J_2$.
\end{enumerate}
\end{lema}
\begin{proof}
    If $\{i,j\}=\{1,2\}$, using the $K$-linear disjointness of $L_i$ with $L_j$ and $M_i$, we have that $\widetilde{L_i}\cap L=(L_iM_i)\cap(L_iL_j)=L_i(L_j\cap M_i)=L_i$, the last equality due to $L_1/K$ and $L_2/K$ being strongly disjoint. Hence, every $g\in G'$ satisfies $g|_{\widetilde{L_i}}\in\mathrm{Gal}(\widetilde{L_i}/L_i)$. Then, by means of the monomorphism $G\hookrightarrow G_1\times G_2$, $G'$ is embedded in $G_1'\times G_2'$. Likewise, we have that $\widetilde{L_i}\cap M=M_i$ for $i\in\{1,2\}$, and $J$ is embedded in $J_1\times J_2$. 
    
    Now, call $n_i=[L_i:K]$ for $i\in\{1,2\}$ and $n=[L:K]$. Since $L_1/K$ and $L_2/K$ are linearly disjoint, $n=n_1n_2$. Moreover, the $K$-linear disjointness of $L_i$ and $M_i$ gives that $n_i=[\widetilde{L}_i:M_i]=|J_i|$, and similarly, the $K$-linear disjointness of $L$ and $M$ yields $n=|J|$. We deduce that $|J|=|J_1\times J_2|$, and \eqref{lemaemb2} follows.
\end{proof}

%\begin{rmk}\normalfont In the proof of Lemma \ref{lema:embedding} we have not used the strong disjointness. In fact, we have proved that the result holds simply by assuming that $L_1/K$ and $L_2/K$ are linearly disjoint.
%\end{rmk}

Let $\lambda\colon G\longrightarrow\mathrm{Perm}(G/G')$ be the left translation map of $L/K$, and let $\lambda_i\colon G_i\longrightarrow\mathrm{Perm}(G_i/G_i')$ be the left translation map of $L_i/K$ for $i\in\{1,2\}$. We will prove that the definition of $\lambda$ can be recovered from the ones of $\lambda_1$ and $\lambda_2$. Since $G'$ is embedded in $G_1'\times G_2'$, there is a well defined map $$\begin{array}{rccl}
    \psi \colon & G/G' & \longrightarrow & G_1/G_1'\times G_2/G_2' \\
    & g_1g_2G' & \longmapsto & (g_1G_1',g_2G_2')
\end{array}$$ 

\begin{lema}\label{lem:psibij} The map $\psi$ is a bijection.
\end{lema}
\begin{proof}
First, we will check that $\psi$ is surjective. Let $(g_1G_1',g_2G_2')\in G_1/G_1'\times G_2/G_2'$. We know that the elements of $J_i$ form a transversal for $G_i/G_i'$, that is: if $J_i=\{\sigma_1^{(i)},\dots,\sigma_{n_i}^{(i)}\}$, then $G_i/G_i'=\{\sigma_1^{(i)}G_i',\dots,\sigma_{n_i}^{(i)}G_i'\}$. Then for each $i\in\{1,2\}$ there are a unique $j_i\in\{1,\dots,n_i\}$ and some $g_i'\in G_i'$ such that $g_i=\sigma_{j_i}^{(i)}g_i'$. Now, using the operation of the direct product $G_1\times G_2$, $$g_1g_2=\sigma_{j_1}^{(1)}g_1'\sigma_{j_2}^{(2)}g_2'=\sigma_{j_1}^{(1)}\sigma_{j_2}^{(2)}g_1'g_2',$$ and since $J$ is isomorphic to $J_1\times J_2$ under the embedding $G\hookrightarrow G_1\times G_2$, we have that $\sigma_{j_1}^{(1)}\sigma_{j_2}^{(2)}\in J\subseteq G$. Now, $$\psi(\sigma_{j_1}^{(1)}\sigma_{j_2}^{(2)}G')=(\sigma_{j_1}^{(1)}G_1',\sigma_{j_2}^{(2)}G_2')=(g_1G_1',g_2G_2'),$$ and the surjectivity follows. 

The $K$-linear disjointness of $L_1$ and $L_2$ gives that the domain and codomain have the same (finite) number of elements, so $\psi$ is a bijection.
\end{proof}

Let us identify $G/G'$ with $G_1/G_1'\times G_2/G_2'$ by writing $g_1g_2G'=g_1G_1'g_2G_2'$ for $g_1\in G_1,g_2\in G_2$.

\begin{lema}\label{lema:declambda} With the notation above, $\lambda=\iota\circ\chi$, where \begin{equation}\label{eq:declambda2}\begin{array}{rcccr}
    \chi\colon & G & \longrightarrow & \mathrm{Perm}(G_1/G_1')\times\mathrm{Perm}(G_2/G_2'), &\\
     & g_1g_2 & \longmapsto & (\lambda_1(g_1),\lambda_2(g_2)), &\\
     \\
    \iota\colon & \mathrm{Perm}(G_1/G_1')\times\mathrm{Perm}(G_2/G_2') & \longrightarrow & \mathrm{Perm}(G/G'),&\\
     & (\varphi_1,\varphi_2) & \longmapsto & g_1g_2G'\mapsto\varphi_1(g_1G_1')\varphi_2(g_2G_2').&
\end{array}\end{equation} Moreover, both $\chi$ and $\iota$ are group monomorphisms.
\end{lema}
\begin{rmk}\normalfont If we remove the identification of $G/G'$ with its image by $\psi$, then the definition of $\iota(\varphi_1,\varphi_2)$ for $\varphi_i\in\mathrm{Perm}(G_i/G_i')$ is $\iota(\varphi_1,\varphi_2)(g_1g_2G')=\psi^{-1}(\varphi_1(g_1G_1'),\varphi_2(g_2G_2'))$, which is well defined because of the bijectivity of $\psi$.
\end{rmk}
\begin{proof}
Fix $g=g_1g_2\in G$ for $g_i\in G_i$. Given $hG'\in G/G'$, write $h=h_1h_2$ with $h_i\in G_i$, so that $h_1h_2G'=h_1G_1'h_2G_2'$. Then we have
\begin{equation*}
    \begin{split}
        \iota\circ\chi(g)(hG')&=\iota(\lambda_1(g_1),\lambda_2(g_2))(h_1h_2G')\\&=\lambda_1(g_1)(h_1G_1')\lambda_2(g_2)(h_2G_2')\\&=(g_1h_1G_1')(g_2h_2G_2')=g_1g_2h_1h_2G'=\lambda(g)(hG').
    \end{split}
\end{equation*}

That $\chi$ is a monomorphism follows immediately from the fact that so are $\lambda_1$ and $\lambda_2$. As for $\iota$, let $(\varphi_1,\varphi_2)\in\mathrm{Ker}(\iota)$, so that $\iota(\varphi_1,\varphi_2)=\mathrm{Id}_{G/G'}$. Given $(g_1,g_2)\in G_1\times G_2$, we know from Lemma \ref{lem:psibij} that there is $g\in G$ such that $\psi(gG')=(g_1G_1',g_2G_2')$, that is $gG'=g_1G_1'g_2G_2'$. Then $\iota(\varphi_1,\varphi_2)(gG')=\varphi_1(g_1G_1')\varphi_2(g_2G_2')$, but also $\iota(\varphi_1,\varphi_2)(gG')=gG'=g_1G_1'g_2G_2'$. Now, the injectivity of $\psi$ gives $\varphi_1(g_1G_1')=g_1G_1'$ and $\varphi_2(g_2G_2')=g_2G_2'$. Since $g_1$ and $g_2$ are arbitrary, we conclude that $\varphi_i=\mathrm{Id}_{G_i/G_i'}$.
\end{proof}

\begin{pro}\label{pro:prodhgstrsubg} If $N_i$ is a regular subgroup of $\mathrm{Perm}(G_i/G_i')$ normalized by $\lambda_i(G_i)$ for $i\in\{1,2\}$, then $N=\iota(N_1\times N_2)$ is a regular subgroup of $\mathrm{Perm}(G/G')$ normalized by $\lambda(G)$.
\end{pro}
\begin{proof}
First, we prove that $N$ is regular. Since $N_1$ and $N_2$ are regular, we have that $|N|=|N_1||N_2|=|G_1/G_1'||G_2/G_2'|=|G/G'|$, so it is enough to check that the action of $N$ on $G/G'$ is transitive. Let $gG',\,hG'\in G/G'$ and write $g=g_1g_2$, $h=h_1h_2$ with $g_i,h_i\in G_i$. We know that $N_i$ acts transitively on $G_i/G_i'$, so there is $\eta_i\in N_i$ such that $\eta_i(g_iG_i')=h_iG_i'$ for $i\in\{1,2\}$. Therefore, $\iota(\eta_1,\eta_2)(gG')=\eta_1(g_1G_1')\eta_2(g_2G_2')=(h_1G_1')(h_2G_2')=hG'$.

Now, let us prove that $N$ is normalized by $\lambda(G)$. If $(\eta,\mu)\in N_1\times N_2$ and $g=g_1g_2\in G$, we have that
$$\chi(g)(\eta,\mu)\chi(g)^{-1}=(\lambda_1(g_1)\eta\lambda(g_1)^{-1},\lambda_2(g_2)\mu\lambda_2(g_2)^{-1})\in N_1\times N_2.$$ Thus, since $\lambda=\iota\circ\chi$, we have
$$\lambda(g)\iota(\eta,\mu)\lambda(g)^{-1}=\iota(\lambda(g)(\eta,\mu)\lambda(g)^{-1})\in N.$$
\end{proof}

\subsection{The product Hopf-Galois structure}

Applying the correspondence provided by the Greither-Pareigis theorem to the statement of Proposition \ref{pro:prodhgstrsubg}, we obtain a Hopf-Galois structure on $L/K$ from Hopf-Galois structures on $L_1/K$ and $L_2/K$.

\begin{defi}\label{defi:prodhgstr} Let $L_1/K$ and $L_2/K$ be strongly disjoint almost classically Galois extensions and, for $i\in\{1,2\}$, let $H_i$ be the Hopf-Galois structure on $L_i/K$ given by a permutation subgroup $N_i$. For $L=L_1L_2$, the Hopf-Galois structure on $L/K$ given by $N=\iota(N_1\times N_2)$ will be referred to as the \textbf{product Hopf-Galois structure} on $L/K$ from $H_1$ and $H_2$.
\end{defi}

The name for this construction will be justified after Proposition \ref{pro:prodhgstr} below, where we will prove that the product Hopf-Galois structure $H$ from $H_1$ and $H_2$ can be seen as the tensor product of $H_1$ and $H_2$, both at the level of Hopf algebras and at the level of actions.

We now consider the product of almost classically Galois structures.

\begin{pro}\label{pro:prodalmostclassical} Let $L_1/K$ and $L_2/K$ be strongly disjoint almost classically Galois extensions with complements $M_1$ and $M_2$, and call $L=L_1L_2$ and $M=M_1M_2$. If $H_i$ is an almost classically Galois structure on $L_i/K$, then the product Hopf-Galois structure $H$ from $H_1$ and $H_2$ on $L/K$ is almost classically Galois. Moreover, if $H_i$ is the almost classically Galois structure corresponding to $M_i$ for each $i\in\{1,2\}$, then $H$ is the almost classically Galois structure on $L/K$ corresponding to $M$.
\end{pro}
\begin{proof}
For each $i\in\{1,2\}$, let $J_i=\mathrm{Gal}(\widetilde{L_i}/M_i)$, $G_i'=\mathrm{Gal}(\widetilde{L_i}/L_i)$ and $\lambda_i\colon G_i\longrightarrow\mathrm{Perm}(G_i/G_i')$ be the left translation map for $L_i/K$. Let $N_i$ be the subgroup of $\mathrm{Perm}(G_i/G_i')$ giving an almost classically Galois structure $H_i$ on $L_i/K$, so that $N_i^{\mathrm{opp}}\subset\lambda_i(G_i)$. Let $N=\iota(N_1\times N_2)$, which is the regular subgroup of $\mathrm{Perm}(G/G')$ giving the product Hopf-Galois structure $H$ on $L/K$ from $H_1$ and $H_2$. We need to check that $N^{\mathrm{opp}}\subset\lambda(G)$. 

First, note that $N^{\mathrm{opp}}=\iota(N_1^{\mathrm{opp}}\times N_2^{\mathrm{opp}})$. Indeed, each element in the right side member is centralized by $N$, so $\iota(N_1^{\mathrm{opp}}\times N_2^{\mathrm{opp}})\subseteq N^{\mathrm{opp}}$. Since $N_i^{\mathrm{opp}}$ is a regular subgroup of $\mathrm{Perm}(G_i/G_i')$, Proposition \ref{pro:prodhgstrsubg} gives that $\iota(N_1^{\mathrm{opp}}\times N_2^{\mathrm{opp}})$ is a regular subgroup of $\mathrm{Perm}(G/G')$. But $N^{\mathrm{opp}}$ also is, so they have the same order and then the equality holds. 

Since $H_i$ is almost classically Galois, $N_i^{\mathrm{opp}}\subset\lambda_i(G_i)$. Then, given $\eta\in N^{\mathrm{opp}}$, there are $g_i\in G_i$ such that $\eta=\iota(\lambda_1(g_1),\lambda_2(g_2))$. Now, given $h_i\in G_i$, we have that $$\eta(h_1h_2G')=\lambda(g_1)(h_1G_1')\lambda(g_2)(h_2G_2')=(g_1h_1G_1')(g_2h_2G_2')=g_1g_2h_1h_2G'=\lambda(g_1g_2)(h_1h_2G').$$ Then $\eta=\lambda(g_1g_2)\in\lambda(G)$. We obtain that $N^{\mathrm{opp}}\subset\lambda(G)$, which proves that $H$ is almost classically Galois.

Now, let us assume that $H_i$ corresponds to $M_i$, so $N_i=\lambda_i(J_i)^{\mathrm{opp}}$. Then $N=\iota(\lambda_1(J_1)^{\mathrm{opp}}\times\lambda_2(J_2)^{\mathrm{opp}})$. Let us write $J=\mathrm{Gal}(\widetilde{L}/M)$, $G'=\mathrm{Gal}(\widetilde{L}/L)$ and $\lambda\colon G\longrightarrow \mathrm{Perm}(G/G')$. It is easy to check that $\lambda(J)=\iota(\lambda_1(J_1)\times\lambda_2(J_2))$, so $N=\lambda(J)^{\mathrm{opp}}$. Hence $H$ is the almost classically Galois structure on $L/K$ corresponding to $M$.
\end{proof}

\begin{rmk}\normalfont It is immediate from Proposition \ref{pro:prodalmostclassical} that under the assumption that $H_i$ is the almost classically Galois structure corresponding to $M_i$ for each $i\in\{1,2\}$, a Hopf-Galois structure $H$ on $L/K$ is the product Hopf-Galois structure from $H_1$ and $H_2$ if and only if it is the almost classically Galois structure on $L/K$, due to the uniqueness of such structures.
\end{rmk}

As noticed at the end of Section \ref{sect:almostclassic}, the almost classically Galois structure corresponding to a complement in a Galois extension is its classical Galois structure. Thus, we obtain the following.

\begin{coro}\label{coro:prodclassical} Let $L_1/K$ and $L_2/K$ be two linearly disjoint Galois extensions. The product Hopf-Galois structure on $L/K$ of the classical Galois structures on $L_1/K$ and $L_2/K$ is the classical Galois structure on $L/K$.
\end{coro}

\subsubsection*{The likeness with induced Hopf-Galois structures}

In this section we compare the notion of product Hopf-Galois structure with the one of induced Hopf-Galois structure that we discussed in Section \ref{sect:inducedhgstr}. 

Let $L/K$ be a Galois extension with group $G=J\rtimes G'$, and let $E=L^{G'}$, $F=L^J$. The extension $E/K$ is almost classically Galois because its Galois closure $\widetilde{E}$ satisfies $\widetilde{E}\subseteq L\cong E\otimes_KF$ and we apply Theorem \ref{charactalmostclassical} \eqref{almostclassiccond1}. The Galois complement of $E/K$ is necessarily contained in $F$. Hence, the extensions $E/K$ and $F/K$ are not strongly disjoint unless the extension $L/K$ is Galois. Then, it only makes sense to consider both product Hopf-Galois structures and induced Hopf-Galois structures in compositums of linearly disjoint Galois extensions. In that case, both notions are the same.

\begin{pro} Let $E/K$ and $F/K$ be linearly disjoint Galois extensions with groups $J$ and $G'$ respectively, and let $L=EF$. Then the induced Hopf-Galois structures on $L/K$ are the product Hopf-Galois structures on $L/K$ from Hopf-Galois structures on $E/K$ and $F/K$.
\end{pro}
\begin{proof}
Since $E/K$ and $F/K$ are linearly disjoint, by Proposition \ref{pro:prodgalois}, the extension $L/K$ is Galois with Galois group $G$ isomorphic to the direct product $J\times G'$. Then, it makes sense to consider induced Hopf-Galois structures on $L/K$. Moreover, $E/K$ and $F/K$ are strongly disjoint, so it makes sense to consider product Hopf-Galois structures on $L/K$.

By the Greither-Pareigis theorem, the Hopf-Galois structures on $L/K$ are in bijective correspondence with the regular subgroups of $\mathrm{Perm}(G)$ normalized by the image of the left translation map $\lambda\colon G\longrightarrow\mathrm{Perm}(G)$. It is enough to prove that such a subgroup is as in Proposition \ref{pro:ngivesinduced} if and only if it is as in Proposition \ref{pro:prodhgstrsubg}. In order to do so, we will see that the decompositions of $\lambda$ in \eqref{eq:declambda} and \eqref{eq:declambda2} are actually the same (up to identifications).

Let $\iota$ and $\chi$ be the maps at \eqref{eq:declambda}. Note that since $G\cong J\times G'$, $J\cong\mathrm{Gal}(L/F)$ and $G'\cong\mathrm{Gal}(L/E)$. Then, we can write $\chi\colon G\longrightarrow\mathrm{Perm}(J)\times\mathrm{Perm}(G')$ and $\iota\colon\mathrm{Perm}(J)\times\mathrm{Perm}(G')\longrightarrow\mathrm{Perm}(G)$ also in this context. Recall that these maps are defined as $$\chi(\sigma\tau)=(\lambda_c(\sigma\tau),\lambda'(\tau)),$$ $$\iota(\varphi,\psi)(\sigma\tau)=\varphi(\sigma)\psi(\tau),$$ where $\lambda'\colon G'\longrightarrow\mathrm{Perm}(G')$ is the left regular representation of $G'$ and the map $\lambda_c\colon G\longrightarrow\mathrm{Perm}(J)$ is obtained from the left translation map $\overline{\lambda}\colon G\longrightarrow\mathrm{Perm}(G/G')$ associated to $E/K$ (recall that we can apply the Greither-Pareigis theorem as if $L=\widetilde{E}$) by considering the isomorphism $J\cong G/G'$, since $J$ is a transversal of $G'$ in $G$.

Now, let $\iota'$ and $\chi'$ be the maps at \eqref{eq:declambda2}. The isomorphisms $J\cong G/G'$ and $G'\cong G/J$ allow us to write $\chi'\colon G\longrightarrow\mathrm{Perm}(J)\times\mathrm{Perm}(G')$ and $\iota'\colon\mathrm{Perm}(J)\times\mathrm{Perm}(G')\longrightarrow\mathrm{Perm}(G)$. With this identification, $\iota'$ is defined as $$\iota'(\varphi,\psi)(\sigma\tau)=\varphi(\sigma)\psi(\tau),$$ and we obtain $\iota=\iota'$. Then $\lambda=\iota\circ\chi=\iota\circ\chi'$, and since $\iota$ is a monomorphism, the equality $\chi=\chi'$ follows.

Then, induced Hopf-Galois structures and product Hopf-Galois structures on $L/K$ are built in the same way.
\end{proof}

From the previous discussion it holds that as soon as one of the extensions we consider is not Galois, at most one of the two notions apply. However, there are still some similarities. Namely, we have an analogue of Proposition \ref{pro:prodindhgstr} for product Hopf-Galois structures.

\begin{pro}\label{pro:prodhgstr} Let $L/K$ be the compositum of two strongly disjoint almost classically Galois extensions $L_1/K$ and $L_2/K$. Let $H_i$ be a Hopf-Galois structure on $L_i/K$ and let $H$ be the product Hopf-Galois structure on $L/K$ from $H_1$ and $H_2$. Then:
\begin{itemize}
    \item[1.] $H\cong H_1\otimes_KH_2$ as $K$-algebras.
    \item[2.] If $h_i\in H_i$ and $\alpha_i\in L_i$ for $i\in\{1,2\}$, then $(h_1h_2)\cdot(\alpha_1\alpha_2)=(h_1\cdot \alpha_1)(h_2\cdot \alpha_2)$.
\end{itemize}
\end{pro}
\begin{proof}
Let $N_i$ be the permutation subgroup corresponding to $H_i$ under the Greither-Pareigis correspondence, so that $H_i=\widetilde{L_i}[N_i]^{G_i}$. Then $H=\widetilde{L}[N]^G$, where $N=\iota(N_1\times N_2)$.
\begin{itemize}
    \item[1.] First of all, given $g=g_1g_2\in G$ with $g_1\in G_1$ and $g_2\in G_2$, $\eta_1\in N_1$ and $\eta_2\in N_2$, we have that $$g\cdot\iota(\eta_1,\eta_2)=\lambda(g)\iota(\eta_1,\eta_2)\lambda(g^{-1})=\iota(\chi(g)(\eta_1,\eta_2)\chi(g^{-1})),$$ and since $\chi(g)=(\lambda_1(g_1),\lambda_2(g_2))$, we have $$\chi(g)(\eta_1,\eta_2)\chi(g^{-1})=(\lambda_1(g_1)\eta_1\lambda_1(g_1^{-1}),\lambda_2(g_2)\eta_2\lambda_2(g_2^{-1}))=(g_1\cdot\eta_1,g_2\cdot\eta_2).$$ We conclude that $$g\cdot\iota(\eta_1,\eta_2)=\iota(g_1\cdot\eta_1,g_2\cdot\eta_2).$$
    
     Write $N_1=\{\eta_i^{(1)}\}_{i=1}^{n_1}$ and $N_2=\{\eta_j^{(2)}\}_{j=1}^{n_2}$. Let us consider the map $f\colon H_1\otimes_KH_2\longrightarrow \widetilde{L}[N]$ defined by $$f(h_1\otimes h_2)=\sum_{i=1}^{n_1}\sum_{j=1}^{n_2}h_i^{(1)}h_j^{(2)}\iota(\eta_i^{(1)},\eta_j^{(2)}),$$ where $h_1=\sum_{i=1}^{n_1}h_i^{(1)}\eta_i^{(1)}\in H_1$ and $h_2=\sum_{j=1}^{n_2}h_j^{(2)}\eta_j^{(2)}\in H_2$, with $h_i^{(1)}\in\widetilde{L}_1$ and $h_j^{(2)}\in\widetilde{L}_2$. It is clear that $f$ is a morphism of $K$-algebras. Let us check that it is injective. Indeed, if $h_1\otimes h_2\in\mathrm{Ker}(f)$, this means that $$\sum_{i=1}^{n_1}\sum_{j=1}^{n_2}h_i^{(1)}h_j^{(2)}\iota(\eta_i^{(1)},\eta_j^{(2)})=0.$$ Regarding this as an equality in the $\widetilde{L}$-vector space $\widetilde{L}[N]$, since $\{\iota(\eta_i^{(1)},\eta_j^{(2)})\}_{1\leq i\leq n_1,1\leq j\leq n_2}$ is an $\widetilde{L}$-basis of $\widetilde{L}[N]$, we obtain that $h_i^{(1)}h_j^{(2)}=0$ for every $1\leq i\leq n_1$ and every $1\leq j\leq n_2$. This gives immediately that $h_1\otimes h_2=0$.
     
     It is enough to prove that $f(H_1\otimes_K H_2)\subseteq H$, as we have that $\mathrm{dim}_K(H)=\mathrm{dim}_K(H_1\otimes_KH_2)=\mathrm{dim}_K(f(H_1\otimes_KH_2))$ and then it will follow that $f$ is an isomorphism over its image. Let $h_1\in H_1$ and $h_2\in H_2$ be as before and let $g=g_1g_2\in G$. Then 
    \begin{equation*}
        \begin{split}
            g\cdot(f(h_1\otimes h_2))&=\sum_{i=1}^{n_1}\sum_{j=1}^{n_2}g_1(h_i^{(1)})g_2(h_j^{(2)})\iota(g_1\cdot\eta_i^{(1)},g_2\cdot\eta_j^{(2)})
            \\&=f(g_1\cdot h_1\otimes g_2\cdot h_2)=f(h_1\otimes h_2),
        \end{split}
    \end{equation*} which proves that $f(H_1\otimes_KH_2)\subseteq H$ as desired.
    
    \item[2.] Note that in this statement we have identified $H$ with $H_1\otimes_KH_2$ via the isomorphism $f$ in the first statement and $L$ with $L_1\otimes_K L_2$ via the canonical map (since $L_1$ and $L_2$ are $K$-linearly disjoint). Write $h_1$ and $h_2$ as in the previous proof. Following the general description at \eqref{hopfactiongp} of a Hopf action, we can write $$f(h_1\otimes h_2)\cdot(\alpha_1\alpha_2)=\sum_{i=1}^{n_1}\sum_{j=1}^{n_2}h_i^{(1)}h_j^{(2)}\iota(\eta_i^{(1)},\eta_j^{(2)})^{-1}(1_GG')(\alpha_1\alpha_2).$$ For each $1\leq i\leq n_1$ and $1\leq j\leq n_2$, we have $\iota(\eta_i^{(1)},\eta_j^{(2)})^{-1}(1_GG')=(\eta_i^{(1)})^{-1}(1_{G_1}G_1')(\eta_j^{(2)})^{-1}(1_{G_2}G_2')$. Let $k_i^{(1)}\in G_1$ and $k_j^{(2)}\in G_2$ such that $k_i^{(1)}G_1'=(\eta_i^{(1)})^{-1}(1_{G_1}G_1')$ and $k_j^{(2)}G_2'=(\eta_j^{(2)})^{-1}(1_{G_2}G_2')$. Then $$(\eta_i^{(1)})^{-1}(1_{G_1}G_1')(\eta_j^{(2)})^{-1}(1_{G_2}G_2')=(k_i^{(1)}k_j^{(2)})G',$$ and hence \begin{equation}\label{eq1}
        \begin{split}
            f(h_1\otimes h_2)\cdot(\alpha_1\alpha_2)&=\sum_{i=1}^{n_1}\sum_{j=1}^{n_2}h_i^{(1)}h_j^{(2)}(\eta_i^{(1)})^{-1}(1_{G_1}G_1')(\eta_j^{(2)})^{-1}(1_{G_2}G_2')(\alpha_1\alpha_2)\\&=\sum_{i=1}^{n_1}\sum_{j=1}^{n_2}h_i^{(1)}h_j^{(2)}[k_i^{(1)}k_j^{(2)}(\alpha_1\alpha_2)].
        \end{split}
    \end{equation} Following the paragraph after Proposition \ref{pro:prodgalois}, we have $$k_i^{(1)}k_j^{(2)}(\alpha_1\alpha_2)=(k_i^{(1)}k_j^{(2)})|_{L_1}(\alpha_1)(k_i^{(1)}k_j^{(2)})|_{L_2}(\alpha_2).$$ The strong disjointness conditions on $L_1/K$ and $L_2/K$ imply immediately that $L_1\cap\widetilde{L_2}=K$ and $L_2\cap\widetilde{L_1}=K$, whence $(k_i^{(1)}k_j^{(2)})|_{L_1}=k_i^{(1)}$ and $(k_i^{(1)}k_j^{(2)})|_{L_2}=k_j^{(2)}$. Then $$k_i^{(1)}k_j^{(2)}(\alpha_1\alpha_2)=k_i^{(1)}(\alpha_1)k_j^{(2)}(\alpha_2).$$ Carrying this to \eqref{eq1}, we have \begin{equation*}
        \begin{split}
            f(h_1\otimes h_2)\cdot(\alpha_1\alpha_2)&=\sum_{i=1}^{n_1}\sum_{j=1}^{n_2}h_i^{(1)}h_j^{(2)}k_i^{(1)}(\alpha_1)k_j^{(2)}(\alpha_2)
            \\&=\Big(\sum_{i=1}^{n_1}h_i^{(1)}k_i^{(1)}(\alpha_1)\Big)\Big(\sum_{j=1}^{n_2}h_j^{(2)}k_j^{(2)}(\alpha_2)\Big)
            \\&=\Big(\sum_{i=1}^{n_1}h_i^{(1)}(\eta_i^{(1)})^{-1}(1_{G_1}G_1')(\alpha_1)\Big)\Big(\sum_{j=1}^{n_2}h_j^{(2)}(\eta_j^{(2)})^{-1}(1_{G_2}G_2')(\alpha_2)\Big)
            \\&=(h_1\cdot\alpha_1)(h_2\cdot\alpha_2).
        \end{split}
    \end{equation*}
\end{itemize}
\end{proof}

From now on, we identify $H$ with $H_1\otimes_K H_2$ by means of $f$ by writing $f(h_1\otimes h_2)=h_1h_2$ for $h_1\in H_1$ and $h_2\in H_2$. %Hence, for each $h\in H$ there are unique elements $h_1\in H_1$ and $h_2\in H_2$ such that $h=h_1h_2$. 

Now, we want to find a relation between the matrices describing the action (see Section \ref{sect:hgmodtheory}). In \cite[Theorem 5.10]{gilrioinduced}, it was shown that the matrix of the action of an induced Hopf-Galois structure is the Kronecker product of the matrices of the Hopf-Galois structures from which it is built, up to permutation of rows. In the proof of that result, we only use the fact from Proposition \ref{pro:prodindhgstr} that an induced Hopf-Galois structure is isomorphic to the tensor product of the inducing Hopf-Galois structures. Since the same is true for product Hopf-Galois structures from Proposition \ref{pro:prodhgstr}, we deduce the following.

\begin{pro}\label{pro:prodmatr} Let $L_1/K$ and $L_2/K$ be strongly disjoint almost classically Galois extensions. For $i\in\{1,2\}$, let $H_i$ be a Hopf-Galois structure on $L_i/K$ and let $H$ be the product Hopf-Galois structure of these on $L/K$. Then, there is a permutation matrix $P\in\mathrm{GL}_{n^2}(\mathcal{O}_K)$ such that $$PM(H,L)=M(H_1,L_1)\otimes M(H_2,L_2).$$
\end{pro}
\begin{proof}
It is completely analogous to the proof of \cite[Theorem 5.10]{gilrioinduced}.
\end{proof}

Let us consider the setting in Section \ref{sect:hgmodtheory}. Using the method described therein, we can use the relation between the matrices of the action to obtain an analogue of Proposition \ref{pro:indhgmstr} for product Hopf-Galois structures.

\begin{coro}\label{coro:prodhgmstr} Let $K$ be the fraction field of a PID $\mathcal{O}_K$. Let $L_1/K$ and $L_2/K$ be strongly disjoint almost classically Galois extensions, and assume that they are also arithmetically disjoint. For $i\in\{1,2\}$, let $H_i$ be a Hopf-Galois structure on $L_i/K$ and let $H$ be the product Hopf-Galois structure of these on $L/K$, where $L=L_1L_2$.
\begin{enumerate}
    \item $\mathfrak{A}_H=\mathfrak{A}_{H_1}\otimes_{\mathcal{O}_K}\mathfrak{A}_{H_2}$.
    \item If $\mathcal{O}_{L_i}$ is $\mathfrak{A}_{H_i}$-free for $i\in\{1,2\}$, then $\mathcal{O}_L$ is $\mathfrak{A}_H$-free.
\end{enumerate}
\end{coro}
\begin{proof}
Let $B_i$ be an $\mathcal{O}_K$-basis of $\mathcal{O}_{L_i}$ for $i\in\{1,2\}$. The hypothesis that $L_1/K$ and $L_2/K$ are arithmetically disjoint gives that $\mathcal{O}_L=\mathcal{O}_{L_1}\otimes_{\mathcal{O}_K}\mathcal{O}_{L_2}$, so the product of the bases $B_1$ and $B_2$, consisting in all the possible products of an element of $B_1$ and an element of $B_2$, is an $\mathcal{O}_K$-basis of $\mathcal{O}_L$.
\begin{enumerate}
    \item For $i\in\{1,2\}$, let $D_i$ be a reduced matrix of $M(H_i,L_i)$. By definition, there is a unimodular matrix $U_i\in\mathrm{GL}_{n_i^2}(\mathcal{O}_K)$ such that $U_iM(H_i,L_i)=\begin{pmatrix}D_i \\ \hline \\[-2ex] O\end{pmatrix}$, where $M(H_i,L_i)$ is written by fixing the basis $B_i$ in $L_i$ for $i\in\{1,2\}$. Now, the Kronecker product $U_1\otimes U_2\in\mathrm{GL}_{n^2}(\mathcal{O}_K)$ is a unimodular matrix that satisfies \begin{equation*}
        \begin{split}
            &(U_1\otimes U_2)(M(H_1,L_1)\otimes M(H_2,L_2))=(U_1M(H_1,L_1))\otimes(U_2M(H_2,L_2))=\\&\begin{pmatrix}D_1 \\ \hline \\[-2ex] O\end{pmatrix}\otimes\begin{pmatrix}D_2 \\ \hline \\[-2ex] O\end{pmatrix}=\begin{pmatrix}D_1\otimes D_2 \\ \hline \\[-2ex] O\end{pmatrix}.
        \end{split}
    \end{equation*} By Proposition \ref{pro:prodmatr}, there is a unimodular matrix $P\in\mathrm{GL}_{n^2}(\mathcal{O}_K)$ such that $PM(H,L)=M(H_1,L_1)\otimes M(H_2,L_2)$, where $M(H,L)$ is written by fixing the basis $B$ in $L$. Then, $$[(U_1\otimes U_2)P]M(H,L)=\begin{pmatrix}D_1\otimes D_2 \\ \hline \\[-2ex] O\end{pmatrix},$$ and the matrix $(U_1\otimes U_2)P$ is unimodular. Moreover, $B$ is an integral basis of $L$. Then, $D_1\otimes D_2$ is a reduced matrix of $M(H,L)$. Now, the result is obtained by applying Proposition \ref{pro:basisassocorder}.
    \item[2.] For $i\in\{1,2\}$, let $V_i=\{v_j^{(i)}\}_{j=1}^{n_i}$ be an $\mathcal{O}_K$-basis of $\mathcal{O}_{L_i}$ and let $\gamma_i\in\mathcal{O}_{L_i}$ be an $\mathfrak{A}_{H_i}$-generator of $\mathcal{O}_{L_i}$. Then $\{v_j^{(i)}\cdot\gamma_i\}_{j=1}^{n_i}$ is an $\mathcal{O}_K$-basis of $\mathcal{O}_L$ for $i\in\{1,2\}$. Since $L_1/K$ and $L_2/K$ are arithmetically disjoint, the product of these bases is an $\mathcal{O}_K$-basis of $\mathcal{O}_L$. Now, this is formed by the elements of the product $V_1$ and $V_2$ acting on $\gamma=\gamma_1\gamma_2\in\mathcal{O}_L$, and by the previous part that product is an $\mathcal{O}_K$-basis of $\mathcal{O}_L$. Therefore, $\gamma$ is an $\mathfrak{A}_H$-free generator of $\mathcal{O}_L$ and in particular $\mathcal{O}_L$ is $\mathfrak{A}_H$-free.
\end{enumerate}
\end{proof}

For almost classically Galois extensions with associated rings of integers, the conditions of arithmetic and strong disjointness are not redundant. Clearly, two extensions that are strongly disjoint need not be arithmetically disjoint: a pair of linearly disjoint Galois extensions are automatically strongly disjoint, and there are many examples of such pairs that are not arithmetically disjoint. On the other hand, in the following, we will exhibit two arithmetically disjoint extensions that are not strongly disjoint.

\begin{example}\normalfont Let $L_1=\mathbb{Q}_3(\alpha)$, where $\alpha$ is a root of $f_1(x)=x^3+3x^2+3$, and let $L_2=\mathbb{Q}_3(\sqrt{2})$. At the LMFDB database, the extension $L_1/\mathbb{Q}_3$ corresponds to \href{https://www.lmfdb.org/padicField/3.3.4.4}{p-adic field 3.3.4.4}, while the extension $L_2/\mathbb{Q}_3$ corresponds to \href{https://www.lmfdb.org/padicField/3.2.0.1}{p-adic field 3.2.0.1}. Since the extensions $L_1/\mathbb{Q}_3$ and $L_2/\mathbb{Q}_3$ have coprime degrees, they are linearly disjoint. Moreover, since $L_1/\mathbb{Q}_3$ is ramified and $L_2/\mathbb{Q}_3$ is unramified, they are arithmetically disjoint. The extension $L_1/\mathbb{Q}_3$ is almost classically Galois with complement $M_1=\mathbb{Q}_3(\sqrt{-1})=\mathbb{Q}_3(\sqrt{2})$ (see \cite[Example 3.4]{gildegpdihedral}), and $L_2\cap M_1\neq\mathbb{Q}_3$. We conclude that $L_1/\mathbb{Q}_3$ and $L_2/\mathbb{Q}_3$ are not strongly disjoint.
\end{example}

\section{Kummer theory for Galois extensions}\label{kummergalois}

In this section we will review several well known facts for Kummer Galois extensions. From now on, we establish the following conventions. For each positive integer $m$, we fix a primitive $m$-th root of unity $\zeta_m$ such that if $m_1\mid m_2$, then $\zeta_{m_2}^{\frac{m_2}{m_1}}=\zeta_{m_1}$. Moreover, we write $n$ for a positive integer number and $K$ for a field whose characteristic is coprime to $n$. Let us fix an algebraic closure $\overline{K}$ of $K$, so that $\zeta_n\in\overline{K}$.

Given $a\in K$, the polynomial $x^n-a$ has $n$ different roots in $\overline{K}$, namely $\zeta_n^i\alpha$ with $0\leq i\leq n-1$. Let us assume that the polynomial $x^n-a$ is irreducible over $K$. Then, any of its roots generates a degree $n$ extension of $K$, which we denote by $L=K(\sqrt[n]{a})$, or $L=K(\alpha)$ with $\alpha^n=a\in K$. The field $L$ is determined possibly up to $K$-isomorphism. From now on, each time we consider an element $\alpha\in\overline{L}$ with $\alpha^n=a\in K$, we will assume that the polynomial $x^n-a$ is irreducible over $K$.

Now, assume that $K$ contains the $n$-th roots of unity (equivalently $\zeta_n\in K$), so that $L=K(\sqrt[n]{a})$ does not depend on the choice of the root of $x^n-a$. Then $L/K$ is Galois. The assumption that $x^n-a$ is irreducible gives that $\alpha^k\notin K$ for every $k<n$ and every root $\alpha$ of $x^n-a$. If $G$ is the Galois group of $L/K$, the automorphism $\sigma\in G$ such that $\sigma(\alpha)=\zeta_n\alpha$ has clearly order $n$, and therefore $L/K$ is cyclic. Conversely, if $L/K$ is cyclic with Galois group $G=\langle\sigma\rangle$, we know from Hilbert theorem 90 that there is some $\alpha\in L$ such that $\sigma(\alpha)=\zeta_n\alpha$. Since $\sigma$ is a generator of $G$, it follows that $\alpha^n\in K$ and no smaller power of $\alpha$ is in $K$. In summary, we have the following well known result.

\begin{pro}\label{pro:firstcharactcyclic} Assume that $\zeta_n\in K$. Let $L/K$ be a Galois extension of degree $n$ with Galois group $G$. The following statements are equivalent:
\begin{itemize}
    \item[(i)] $L/K$ is cyclic.
    \item[(ii)] $L=K(\alpha)$ for some $\alpha\in L$ such that $\alpha^n\in K$ and no smaller power of $\alpha$ is in $K$.
\end{itemize}
\end{pro}

This provides a complete characterization of the cyclic extensions of the field $K$ in terms of the adjunction of $n$-th roots of unity. It is possible to extend this result to a characterization of all finite abelian extensions of $K$. It is well known that any finite abelian group is a direct product of cyclic subgroups, and its exponent is the least common multiple of the orders of the cyclic subgroups. This motivates the notion of Kummer extension.

\begin{defi} Let $L/K$ be an extension of fields with characteristic coprime to $n$. Assume that $K$ contains the $n$-th roots of unity. We say that $L/K$ is Kummer with respect to $n$ if $L/K$ is abelian and its Galois group has exponent dividing $n$.
\end{defi}

That is, the Galois group of a Kummer extension with respect to $n$ is a direct product of cyclic groups with exponent dividing $n$. Using the fundamental theorem of Galois theory, the following is proved (see for instance \cite[Theorem 11.4]{morandi}).

\begin{pro}\label{charactkummergalois} Assume that $\zeta_n\in K$. Let $L/K$ be a Galois extension with Galois group $G$. The following statements are equivalent:
\begin{itemize}
    \item[(i)] $L/K$ is Kummer with exponent $n$.
    \item[(ii)] $L=K(\alpha_1,\dots,\alpha_k)$ for some $\alpha_1,\dots,\alpha_k\in L$ such that $\alpha_i^n\in K$ for every $1\leq i\leq k$ and no positive integer smaller than $n$ has this property.
\end{itemize}
\end{pro}

Note that (ii) can be restated by: There exist $a_1,\dots,a_k\in K$ such that $L=K(\sqrt[n]{a_1},\dots,\sqrt[n]{a_k})$ and $n$ is the minimal integer number with the property. Suppose that it is satisfied, so that $L/K$ is Kummer with exponent $n$. Call $L_i=K(\sqrt[n]{a_i})$, $n_i=[L_i:K]$, and let $\alpha_i\in L_i$ with $\alpha_i^n=a_i$. Then there is some $0\leq k\leq n$ such that $\zeta_n^k\alpha_i^{n_i}\in K$ (see \cite[Lemma 3.5]{barreramora1999}), whence $\alpha_i^{n_i}\in K$. Since in addition no smaller power of $\alpha_i$ belongs to $K$, each $L_i/K$ is a cyclic extension of degree exactly $n_i$. We can assume without loss of generality that $\{\alpha_1,\dots,\alpha_k\}$ is a minimal set of Kummer generators for $L/K$, in the sense that none of its proper subsets is a set of Kummer generators. In that case, calling $G_i\coloneqq\mathrm{Gal}(L_i/K)$ for each $1\leq i\leq k$, it is shown that $G=\prod_{i=1}^kG_i$ and $G_i\cong G/H_i$, where $H_i=\prod_{j=1,j\neq i}^kG_i$. Moreover, $i\neq j$ implies that $L_i\cap L_j=L^{H_iH_j}=L^G=K$, so the extensions $L_i/K$ and $L_j/K$ are linearly disjoint. Then, Kummer extensions with exponent $n$ are the compositums of pairwise linearly disjoint cyclic extensions with degree dividing $n$.

Each extension $L/K$ as in (ii) gives rise to a unique finitely generated subgroup of the multiplicative group $K^*/(K^*)^n$. Namely, if $L=K(\sqrt[n]{a_1},\dots,\sqrt[n]{a_k})$, then $\langle a_1,\dots,a_n\rangle$ is a subgroup of $K^*$. Now, multiplying any $a_i$ by an $n$-th power of an element in $K^*$ does not vary the extension $L/K$. Thus, the projection of such a subgroup onto $K^*/(K^*)^n$ is completely determined by $L$. Conversely, each subgroup $B$ of $K^*/(K^*)^n$ is assigned to the extension $L=K(\sqrt[n]{B})$, where $$\sqrt[n]{B}=\{\alpha\in L\,|\,\alpha^n\in B\}.$$ Note that each $\alpha\in\overline{K}$ such that $\alpha^n\in B$ belongs to $L$ because of the assumption that $\zeta_n\in K$.

The correspondence described above is bijective because two extensions of the form $K(\sqrt[n]{a})$, $K(\sqrt[n]{b})$ with $a,b\in K$ are the same if and only if there is some $c\in K^*$ and $r\in\mathbb{Z}_{>0}$ coprime to $n$ such that $a=b^rc^n$ (see \cite[Chapter III, Lemma 3]{casselsfrohlich}). Therefore, from Proposition \ref{charactkummergalois} we recover the following well known result:

\begin{coro}\label{corocorrgalois} Assume that $\zeta_n\in K$. Let $L/K$ be a Galois extension with Galois group $G$. The map $B\mapsto K(\sqrt[n]{B})$ defines a bijective correspondence between:
\begin{itemize}
    \item[1.] Finitely generated subgroups of the multiplicative group $K^*/(K^*)^n$.
    \item[2.] Finite abelian extensions of $K$ with exponent $n$.
\end{itemize}
Within this correspondence, degree $n$ cyclic extensions of $K$ correspond bijectively to cyclic subgroups of $K^*/(K^*)^n$.
\end{coro}

This result is sometimes known as the main theorem of Kummer theory.

\subsection{A characterization of Kummer Galois extensions}

In this section we will rewrite the Kummer condition for a Galois extension $L/K$ in terms of the action of its Galois group. We will start with the cyclic case. We know from Proposition \ref{pro:firstcharactcyclic} that such an extension is of the form $L=K(\alpha)$ with $\alpha^n\in K$ whenever $\zeta_n\in K$, since it has cyclic Galois group $G$. Now, note that for every $\sigma\in G$ there is a unique $0\leq i_{\sigma}\leq n-1$ such that $\sigma(\alpha)=\zeta_n^{i_{\sigma}}\alpha$. This means that the element $\alpha$ is an eigenvector of all the elements $\sigma\in G$, where these are regarded as $K$-endomorphisms of $L$. This property is also a characterization of cyclic extensions of $K$.

\begin{pro}\label{charactkummercyclic} Let $n\in\mathbb{Z}_{>0}$ and let $K$ be a field with characteristic coprime to $n$. Let $L=K(\alpha)$ be a Galois extension of $K$ with group $G$. The following statements are equivalent:
\begin{enumerate}
    \item\label{charactkummercyclic1} $\zeta_n\in K$, $\alpha^n\in K$ and no smaller power of $\alpha$ is in $K$.
    \item\label{charactkummercyclic2} $\alpha$ is an eigenvector of each automorphism $\sigma\in G$ and $|G|=n$.
\end{enumerate}
\end{pro}
\begin{proof}
We have already seen that \eqref{charactkummercyclic1} implies \eqref{charactkummercyclic2}. Conversely, let us assume the situation in \eqref{charactkummercyclic2}, so that $|G|=n$ and for each $\sigma\in G$ there is an element $\lambda_{\sigma}\in K$ such that $\sigma(\alpha)=\lambda_{\sigma}\alpha$. Since $\alpha$ is a primitive element, all the elements $\sigma(\alpha)$ are distinct as $\sigma$ runs through $G$, so the norm of $\alpha$ is $N(\alpha)=\prod_{\sigma\in G}\sigma(\alpha)=\prod_{\sigma\in G}\lambda_{\sigma}\alpha^n$, which obviously belongs to $K$. Since the $\lambda_{\sigma}$ also do, we obtain that $\alpha^n\in K$. In other words, $L=K(\alpha)$ where $\alpha^n\in K$. Moreover, the condition that $|G|=n$ ensures that the minimal polynomial of $\alpha$ has degree $n$, so no smaller power of $\alpha$ belongs to $K$. Let us check that $\zeta_n\in K$. Indeed, the minimal polynomial of $\alpha$ over $K$ is $f(x)=x^n-a$ with $a\coloneqq\alpha^n$, and hence the conjugates of $\alpha$ are $\alpha,\,\zeta_n\alpha,\dots,\,\zeta_n^{n-1}\alpha$. Thus, there is a unique $\sigma_1\in G$ such that $\sigma_1(\alpha)=\zeta_n\alpha$, so $\zeta_n=\lambda_{\sigma_1}\in K$. %there is a unique $0\leq i_{\sigma}\leq n-1$ such that $\sigma(\alpha)=\zeta_n^{i_{\sigma}}\alpha$. But $\sigma(\alpha)=\lambda_{\sigma}\alpha$, so $\zeta_n^{i_{\sigma}}=\lambda_{\sigma}\in K$ for every $\sigma\in G$. In particular, $\zeta_n\in K$.
\end{proof}

We will refer to a primitive element $\alpha$ as in Proposition \ref{charactkummercyclic} as a Galois eigenvector or $G$-eigenvector. From the proof of this result, we see that the eigenvalues are $n$-th roots of unity.

Now, we use the same idea to characterize arbitrary Kummer extensions in terms of the Galois action. We have seen that a Kummer extension is a product of cyclic extensions, each of which has a Galois eigenvector as a primitive element. Accordingly, we will prove that Kummer extensions are those with a finite generating set of Galois eigenvectors.

\begin{lema}\label{prodgaloiseig} Let $L_1/K$ and $L_2/K$ be Galois extensions with Galois groups $G_1$ and $G_2$. Let $L=L_1L_2$ and let $G=\mathrm{Gal}(L/K)$. If $\alpha_i\in L_i$ is a $G_i$-eigenvector for $i\in\{1,2\}$, then $\alpha_1\alpha_2$ is a $G$-eigenvector in $L$. Consequently, the union of generating sets of $G_k$-eigenvectors for $L_k/K$ is a generating set of $G$-eigenvectors for $L/K$.
\end{lema}
\begin{proof}
By assumption we have that for each $g_i\in G_i$ there is $\lambda_{g_i}\in K$ such that $g_i(\alpha_i)=\lambda_{g_i}\alpha_i$. From Proposition \ref{pro:prodgalois}, for each $g\in G$ there are unique elements $g_1\in G_1$ and $g_2\in G_2$ such that $g|_{L_i}=g_i$, $i\in\{1,2\}$. Moreover, $G$ acts on $L$ by means of the action of $G_i$ on $L_i$, $i\in\{1,2\}$. %Note that this makes sense because $g_1|_{L_1\cap L_2}=g_2|_{L_1\cap L_2}$. 
Hence, $$g(\alpha_1\alpha_2)=g_1(\alpha_1)g_2(\alpha_2)=\lambda_{g_1}\lambda_{g_2}\alpha_1\alpha_2$$ with $\lambda_{g_1}\lambda_{g_2}\in K$. The last sentence follows directly from the fact that any $G_k$-eigenvector for some $k\in\{1,2\}$ is also a $G$-eigenvector, since $1$ is always a Galois eigenvector.
\end{proof}

\begin{teo}\label{charactkummereig} Let $n\in\mathbb{Z}_{>0}$ and let $K$ be a field with characteristic coprime to $n$. Let $L=K(\alpha_1,\dots,\alpha_k)$ be a finite Galois extension of $K$ with group $G$. The following statements are equivalent:
\begin{enumerate}
    \item\label{charactkummereig1} $\zeta_n\in K$, $\alpha_i^n\in K$ for all $1\leq i\leq k$ and $n$ is minimal for this property.
    \item\label{charactkummereig2} $\{\alpha_1,\dots,\alpha_k\}$ is a generating set of $G$-eigenvectors for $L/K$ and $\mathrm{exp}(G)=n$.
\end{enumerate}
\end{teo}
\begin{proof}
Call $L_i\coloneqq K(\alpha_i)$ and $n_i\coloneqq[L_i:K]$ for each $1\leq i\leq k$.

Assume the situation in \eqref{charactkummereig1}. Then Proposition \ref{charactkummergalois} gives that $L/K$ is Kummer with exponent $n$, and $\{\alpha_1,\dots,\alpha_k\}$ is a set of Kummer generators. Moreover, applying Proposition \ref{charactkummercyclic} at each $L_i/K$, $\alpha_i$ is a $G_i$-eigenvector for each $1\leq i\leq k$. %Then, we can assume without loss of generality that $\{\alpha_1,\dots,\alpha_k\}$ is a minimal set of Kummer generators for $L/K$. Hence $G=\prod_{i=1}^kG_i$ where $G_i\coloneqq\mathrm{Gal}(L_i/K)$ for every $1\leq i\leq k$ (see the paragraph after Proposition \ref{charactkummergalois}). 
By Lemma \ref{prodgaloiseig}, $\{\alpha_1,\dots,\alpha_k\}$ is a generating set of $G$-eigenvectors for $L/K$.

Now, suppose that $L/K$ satisfies \eqref{charactkummereig2}. Given $1\leq i\leq k$, for each $\sigma\in G$ there is a unique $\lambda_{\sigma,i}\in K$ such that $\sigma(\alpha_i)=\lambda_{\sigma,i}\alpha_i$. Then each $L_i/K$ is a Galois extension whose Galois group acts on $L_i$ by restriction of the action of $G$ on $L$. Therefore, given $1\leq i\leq k$, we have that $\sigma_i(\alpha_i)=\lambda_{\sigma_i,i}\alpha_i$ for every $\sigma_i\in G_i\coloneqq\mathrm{Gal}(L_i/K)$. Applying Proposition \ref{charactkummercyclic} at each $L_i/K$, we have that $\zeta_{n_i}\in K$, $\alpha_i^{n_i}\in K$ and no smaller power of $\alpha_i$ belongs to $K$, that is, $L_i/K$ is cyclic of degree $n_i$. By Proposition \ref{pro:prodgalois}, $G$ is embedded in the direct product of the groups $G_i$. But $G$ has exponent $n$ by hypothesis, so $n$ is the least common multiple of the numbers $n_i=|G_i|$. Then it follows that $\alpha_1,\dots,\alpha_k$ are as in \eqref{charactkummereig1}. Let us prove that $\zeta_n\in K$. Let $F$ be the prime field of $K$ (so that $F=\mathbb{Q}$ if $\mathrm{char}(K)=0$ and $F=\mathbb{F}_p$ if $\mathrm{char}(K)=p>0$). Since $\zeta_{n_i}\in K$ for all $1\leq i\leq k$, the compositum of the fields $F(\zeta_{n_i})$ is a subfield of $K$. But this compositum is just $F(\zeta_n)$ because $n$ is the least common multiple of the numbers $n_i$. We conclude that $\zeta_n\in K$ as we wanted to prove. %Let $\zeta=\zeta_{n_1}\dots\zeta_{n_k}\in K$.  We have that $\zeta^n=1$ because $n_i\mid n$ for all $1\leq i\leq k$. %Moreover, since $n$ is the least common multiple of the numbers $n_i$, $\zeta^{n_i}\neq1$ for every $1\leq i\leq k$ such that $n_i<n$. We conclude that $\zeta=\zeta_n\in K$.
\end{proof}

\begin{rmk}\normalfont If $L/K$ is a Galois extension with some $K$-basis of $G$-eigenvectors $\{\alpha_1,\dots,\alpha_k\}$, then the eigenvalues of the elements $\alpha_i$ under the action of $G$ are necessarily $n$-th roots of unity. Indeed, we have seen in the proof of Theorem \ref{charactkummereig} that $G\hookrightarrow G_1\times\dots\times G_k$ for $G_i=\mathrm{Gal}(K(\alpha_i)/K)$, and each $\alpha_i$ is a $G_i$-eigenvector with an $n$-th root of unity as eigenvalue. From the proof of Lemma \ref{prodgaloiseig}, we can see that the eigenvalue of $\alpha_i$ as a $G$-eigenvector is the same as the eigenvalue for $\alpha_i$ as a $G_i$-eigenvector.
\end{rmk}

\section{A Kummer condition for Hopf-Galois extensions}\label{sect:kummerhopfgalois}

Let $K$ be a field with characteristic coprime to $n\in\mathbb{Z}_{>0}$. In Section \ref{kummergalois} we have introduced the notion of Galois eigenvector and we have showed that it can be used to characterize the Kummer condition for Galois extensions. This notion only depends on the Galois group and its Galois action, so it makes sense to define an analogous concept for any Hopf-Galois structure on a given Hopf-Galois extension. 

Let $L/K$ be an $H$-Galois extension of fields and let $\cdot$ be the action of $H$ on $L$. Then the map $$\rho_H\colon H\longrightarrow\mathrm{End}_K(L)$$ defined as $\rho_H(h)(x)=h\cdot x$ is a $K$-linear monomorphism. Indeed, we have that $\rho_H=j\circ \iota$, where $\iota\colon H\longrightarrow L\otimes_KH$ is the canonical inclusion defined by $\iota(h)=1\otimes h$ and $j\colon L\otimes_KH\longrightarrow\mathrm{End}_K(L)$ is the map from Section \ref{secthgtheory}, which is a $K$-linear isomorphism by definition of Hopf-Galois structure.

\begin{defi}\label{defieig} Let $L/K$ be an $H$-Galois extension of fields. We say that an element $\alpha\in L$ is an eigenvector of the action of $H$, or an \textbf{$H$-eigenvector}, if for every $h\in H$ there exists some $\lambda(h)\in K$ such that $$h\cdot\alpha=\lambda(h)\alpha,$$ or equivalently, $\alpha$ is an eigenvector of the $K$-endomorphism $\rho_H(h)$ for every $h\in H$. The element $\lambda(h)$ is called an \textbf{$H$-eigenvalue}.
\end{defi}

When we consider the action of $H$ on different eigenvectors $\alpha$, we will sometimes make explicit mention to the corresponding eigenvalues by denoting $\lambda_{\alpha}(h)\equiv\lambda(h)$.

If $L/K$ is Galois with group $G$ and we write $H_c$ for its classical Galois structure, the notion of $G$-eigenvector in Section \ref{kummergalois} is just the one of $H_c$-eigenvector according to Definition \ref{defieig}.

There are some immediate remarks from Definition \ref{defieig}. The element $\alpha=0$ is always an $H$-eigenvector in $L$, since $h\cdot0=0$ for every $h\in H$. Another trivial example of eigenvector is $\alpha=1$. In this case, we have that $h\cdot1=\epsilon_H(h)$, where $\epsilon_H$ is the counity of $H$ as a $K$-Hopf algebra. Note that if $\alpha\neq0$, the element $\lambda(h)$ is completely determined by $h$. In the sequel we will always assume that $H$-eigenvectors are not zero.

In order to check that an element $\alpha$ is an $H$-eigenvector, it is enough to consider a $K$-basis of $H$. Indeed, if $W=\{w_i\}_{i=1}^n$ is a $K$-basis of $H$, for $h=\sum_{i=1}^nh_iw_i\in H$ we have that $$h\cdot\alpha=\sum_{i=1}^nh_iw_i\cdot\alpha=\Big(\sum_{i=1}^nh_i\lambda(w_i)\Big)\alpha=\lambda(h)\alpha,$$ where $\lambda(h)=\sum_{i=1}^nh_i\lambda(w_i)\in K$.

Under this terminology, Theorem \ref{charactkummereig} states that when $\zeta_n\in K$, a Galois extension $L/K$ is Kummer if and only if it admits some finite generating set of eigenvectors under the action of its Galois group, and therefore under the action of the classical Galois structure on $L/K$. This motivates the following definition.

\begin{defi}\label{defihkummer} Let $L/K$ be an $H$-Galois extension of fields. We say that $L/K$ is \textbf{$H$-Kummer} if it admits some finite generating set of $H$-eigenvectors.
\end{defi}

With this definition, any Kummer extension in the classical sense is Kummer with respect to its classical Galois structure.

Sometimes it will be more convenient to work with $K$-bases of $L$ rather than generating sets for $L/K$. Actually, the existence of such a basis is equivalent to the $H$-Kummer property, due to the following result:

\begin{pro}\label{pro:prodeig} Let $L/K$ be an $H$-Galois extension with some $H$-eigenvector. Then the product of $H$-eigenvectors is also an $H$-eigenvector.
\end{pro}
\begin{proof}
Let $\alpha_1,\alpha_2\in L$ be $H$-eigenvectors and let $h\in H$. Then there are unique elements $\lambda_{\alpha_1}(h),\lambda_{\alpha_2}(h)\in K$ such that $h\cdot\alpha_1=\lambda_{\alpha_1}(h)\alpha_1$ and $h\cdot\alpha_2=\lambda_{\alpha_2}(h)\alpha_2$. Using the Sweedler's notation for $h$, $$h\cdot(\alpha_1\alpha_2)=\sum_{(h)}(h_{(1)}\cdot\alpha_1)(h_{(2)}\cdot\alpha_2)=\sum_{(h)}\lambda_{\alpha_1}(h_{(1)})\lambda_{\alpha_2}(h_{(2)})\alpha_1\alpha_2=\lambda(h)\alpha_1\alpha_2,$$ where $\lambda(h)=\sum_{(h)}\lambda_{\alpha_1}(h_{(1)})\lambda_{\alpha_2}(h_{(2)})\in K$.
\end{proof}

\begin{rmk}\normalfont In general, the sum of $H$-eigenvectors is not an $H$-eigenvector. For example, let $L=\mathbb{Q}(\alpha)$ with $\alpha^3=2$, which admits a unique Hopf-Galois structure $H$ by Byott uniqueness theorem \cite{byottuniqueness}. It is described as follows: Let $G=\mathrm{Gal}(\widetilde{L}/\mathbb{Q})$, where $\widetilde{L}=\mathbb{Q}(\alpha,\zeta_3)$ is the normal closure of $L/\mathbb{Q}$. Let $\sigma\in G$ be defined by $\sigma(\alpha)=\zeta_3\alpha$ and $\sigma(\zeta_3)=\zeta_3$, and let $\tau\in G$ be the automorphism fixing $\alpha$ and taking $\zeta_3$ to its inverse, so that $G$ is generated by $\sigma$ and $\tau$. In the notation of Section \ref{secthgtheory}, let us identify $\sigma$ with $\lambda(\sigma)$. Then we can use the Greither-Pareigis theorem to show that $H=\mathbb{Q}[w]$ where $w=\sqrt{-3}(\sigma-\sigma^2)$ acts on $L$ by means of the Galois action of $\sigma$ on $L$ (this is a straightforward calculation, see for instance \cite[Section 1.2]{cresporiovela2015}). Now, $\alpha$ is an $H$-eigenvector since $w\cdot\alpha=-3\alpha$ and $w^2\cdot\alpha=9\alpha$, and by Proposition \ref{pro:prodeig}, so is $\alpha^2$. However, $w\cdot(\alpha+\alpha^2)=3(-\alpha+\alpha^2)$, and there is no $\mu\in\mathbb{Q}$ such that $w\cdot(\alpha+\alpha^2)=\mu(\alpha+\alpha^2)$.
\end{rmk}

\begin{coro}\label{coro:primitiveeigbasis} Let $L/K$ be an $H$-Galois extension. Assume that there is some primitive element $\alpha$ of $L/K$ which is an $H$-eigenvector. Then $L/K$ admits a $K$-basis of $H$-eigenvectors of $L$.
\end{coro}
\begin{proof}
Let $n=[L:K]$. An easy induction on Proposition \ref{pro:prodeig} shows that if $\alpha\in L$ is an $H$-eigenvector, so is $\alpha^i$ for every positive integer $i$. Thus, when $\alpha$ is in addition a primitive element, we have that $\{1,\alpha,\dots,\alpha^{n-1}\}$ is a $K$-basis of $H$-eigenvectors for $L$.
\end{proof}

\begin{coro} Let $L/K$ be an $H$-Galois extension. Then $L/K$ is $H$-Kummer if and only if $L$ has some $K$-basis of $H$-eigenvectors.
\end{coro}
\begin{proof}
If $L$ has a $K$-basis of $H$-eigenvectors, then this is a generating set of $H$-eigenvectors for $L/K$, so $L/K$ is $H$-Kummer. Conversely, suppose that $L/K$ is $H$-Kummer, and let $\{\alpha_1,\dots,\alpha_k\}$ be a generating set of $H$-eigenvectors for $L/K$. By Proposition \ref{pro:prodeig}, the powers of the elements $\alpha_i$ are also $H$-eigenvectors. Now, since $L$ is the compositum of the fields $K(\alpha_i)$, the elements $\alpha_i$ and their powers form a system of generators for $L$ as a $K$-vector space that are $H$-eigenvectors. Hence, $L$ contains some $K$-basis of $H$-eigenvectors.
\end{proof}

\section{The correspondence with radical extensions}\label{sect:corresprad}

In Section \ref{kummergalois} we established a correspondence between Kummer Galois extensions and radical extensions of a field $K$, and under this correspondence cyclic extensions correspond to simple radical ones. Moreover, we proved that this characterization can be rewritten in terms of the existence of a finite generating system of Galois eigenvectors. In this section we will generalize the results in Section \ref{kummergalois} to the Hopf-Galois setting by means of the notion of $H$-eigenvector introduced in Section \ref{sect:kummerhopfgalois}. As a consequence, we will establish a correspondence between $H$-Kummer extensions and radical extensions of a field $K$ with characteristic coprime to $n\in\mathbb{Z}_{>0}$. This correspondence will not include all $H$-Kummer extensions; in fact it will be defined for a subclass of almost classically Galois extensions of $K$.

\subsection{The case of simple radical extensions}

We will consider first simple radical extensions of a field $K$, i.e. those that are generated by a single $n$-th root of some element in $K$. In the case that $K$ contains the $n$-th roots of unity, this extension is cyclic and, by Proposition \ref{charactkummercyclic}, it corresponds to an extension generated by a single Galois eigenvector. Accordingly, we introduce the following notion.

\begin{defi} Let $L/K$ be an $H$-Galois extension. We say that $L/K$ is $H$-cyclic if it has some primitive element which in addition is an $H$-eigenvector.
\end{defi}

In this part we will prove Theorem \ref{maintheorem1} for $k=1$, which corresponds to simple radical extensions. Namely:

\begin{teo}\label{thm:charactpure} Let $n\in\mathbb{Z}_{>0}$, let $K$ be a field with characteristic coprime to $n$ and let $M=K(\zeta_n)$. 
Let $L=K(\alpha)$ be a finite extension of $K$. The following statements are equivalent:
\begin{enumerate}
    \item\label{thmpure1} $L\cap M=K$, $\alpha^n\in K$ and $n$ is minimal for this property.
    \item\label{thmpure2} $L/K$ is a degree $n$ almost cyclic extension with Galois complement $M$ and $\alpha$ is an $H$-eigenvector of $L$, where $H$ is the almost classically Galois structure on $L/K$ corresponding to $M$.
\end{enumerate}
In particular, the simple radical degree $n$ extensions of $K$ that are linearly disjoint with $M$ are the degree $n$ almost cyclic extensions of $K$ that are $H$-cyclic.
\end{teo}

Note that if in Theorem \ref{thm:charactpure} we impose that $\zeta_n\in K$, then $M=K$ and the condition in \eqref{thmpure2} becomes that $L/K$ is Galois with some primitive element as eigenvector of its Galois action, recovering the statement of Proposition \ref{charactkummercyclic}.

In Theorem \ref{thm:charactpure}, that \eqref{thmpure2} implies \eqref{thmpure1} is an immediate consequence of the following.

\begin{pro}\label{pro:root} Let $n\in\mathbb{Z}_{>0}$ and let $K$ be a field with characteristic coprime to $n$. Let $L/K$ be a degree $n$ almost abelian extension and let $H$ be an almost classically Galois structure corresponding to some Galois complement $M$. Assume that $L/K$ is $H$-cyclic and let $\alpha$ be a primitive element of $L/K$ which is an $H$-eigenvector of $L$. Then $M=K(\zeta_n)$, $\alpha^n\in K$ and no smaller power of $\alpha$ is in $K$. In particular, $L/K$ is an almost cyclic extension.
\end{pro}
\begin{proof}
Since $M$ is the Galois complement of $L/K$, we know that $L$ and $M$ are $K$-linearly disjoint and $\widetilde{L}\cong L\otimes_KM$. On the other hand, that $\alpha$ is an $H$-eigenvector means that for each $h\in H$ there is a unique $\lambda(h)\in K$ such that $h\cdot\alpha=\lambda(h)\alpha$. Call $J=\mathrm{Gal}(\widetilde{L}/M)$, which is abelian by hypothesis. From Proposition \ref{almostclassicstr} we obtain that $H=M[J]^{G'}$, so $M\otimes_K H=M[J]$, and this is the classical Galois structure on $\widetilde{L}/M$. Since elements in $M\otimes_KH$ are $M$-linear combinations of elements in $H$ and $h\cdot\alpha=\lambda(h)\alpha$ for every $h\in H$, we have that $\alpha$ is an eigenvector under the action of $M\otimes_KH$. Therefore, $\alpha$ is an eigenvector with respect to the classical Galois structure on $\widetilde{L}/M$, i.e. a $J$-eigenvector. Since $L/K$ has degree $n$, we know that $|J|=n$. Applying Proposition \ref{charactkummercyclic}, we obtain that $\zeta_n\in M$, $\alpha^n\in M$ and no smaller power of $\alpha$ is in $M$. Thus $\alpha^n\in L\cap M=K$, and if $\alpha^k\in K$ with $1\leq k\leq n$, the fact that $\alpha^k\in M$ implies that $k=n$. Then the minimal polynomial of $\alpha$ over $K$ is $x^n-a$ with $a\coloneqq\alpha^n\in K$, so the conjugates of $\alpha$ are of the form $\zeta_n^k\alpha$, $0\leq k\leq n-1$. We deduce that $\widetilde{L}=L(\zeta_n)=LK(\zeta_n)$. Since in addition $\widetilde{L}=LM$ with $L,M$ $K$-linearly disjoint and $K(\zeta_n)\subseteq M$, we conclude that $M=K(\zeta_n)$.
\end{proof}

It is remarkable that unlike in Theorem \ref{thm:charactpure} \eqref{thmpure2}, in Proposition \ref{pro:root} we do not need to assume that the extension $L/K$ is almost cyclic and that its complement is $K(\zeta_n)$. Instead, this is obtained as a consequence of the assumption that the almost classically Galois extension is almost abelian and $H$-cyclic, where $H$ is the almost classically Galois structure corresponding to the fixed complement.

Next, we prove that the first statement of Theorem \ref{thm:charactpure} implies the second one.

\begin{pro}\label{pro:pureisalmostcyclic} Let $n\in\mathbb{Z}_{>0}$, let $K$ be a field with characteristic coprime to $n$ and let $M=K(\zeta_n)$. Let $L=K(\alpha)$ with $\alpha^n\in K$ and such that $n$ is minimal for this property, and assume that $L\cap M=K$. Then $L/K$ is an almost cyclic extension and $\alpha$ is an $H$-eigenvector of $L$.
\end{pro}
\begin{proof}
The conjugates of $\alpha$ are $\zeta_n^i\alpha$ with $0\leq i\leq n$, so we have that $\widetilde{L}=LM$. In addition $L\cap M=K$ and $M/K$ is Galois, so $L$ and $M$ are $K$-linearly disjoint. Therefore $L/K$ is almost classically Galois with complement $M$. On the other hand, we have that the normal closure of $L$ is $\widetilde{L}=M(\alpha)$ with $\alpha^n\in M$. If $0<k\leq n$ and $\alpha^k\in M$, then $\alpha^k\in L\cap M=K$, so $k=n$. Hence no power of $\alpha$ smaller than $n$ belongs to $M$, and $\zeta_n\in M$. This proves that $\widetilde{L}/M$ is cyclic, that is, $L/K$ is almost cyclic. On the other hand, from Proposition \ref{charactkummercyclic} we see that for each $\sigma\in J$ there is a unique $\lambda_{\sigma}\in M$ such that $\sigma(\alpha)=\lambda_{\sigma}\alpha$. Now, Proposition \ref{almostclassicstr} gives that $H=M[J]^{G'}$. Hence for each $h\in H$ there are $h_1,\dots,h_n\in M$ such that $h=\sum_{i=1}^nh_i\sigma_i$. Then $$h\cdot\alpha=\Big(\sum_{i=1}^nh_i\sigma_i\Big)\cdot\alpha=\Big(\sum_{i=1}^nh_i\lambda_{\sigma_i}\Big)\alpha.$$ Write $\lambda(h)=\sum_{i=1}^nh_i\lambda_{\sigma_i}\in M$. Then $h\cdot\alpha=\lambda(h)\alpha\in L$, so $\lambda(h)\in L\cap M=K$ and $\alpha$ is an $H$-eigenvector of $L$.
\end{proof}

The proof of Theorem \ref{thm:charactpure} is immediate from Propositions \ref{pro:root} and \ref{pro:pureisalmostcyclic}.

If $n$ is a Burnside number (i.e, with the property that $\mathrm{gcd}(n,\varphi(n))=1$), then the condition $L\cap M=K$ is always fulfilled, as $[M:K]$ is always a divisor of $\varphi(n)$, where $\varphi$ is the Euler totient function, and hence $[L:K]$ and $[M:K]$ are coprime. In that case, the almost classically Galois structure $H$ on $L/K$ corresponding to $K(\zeta_n)$ is the unique Hopf-Galois structure on $L/K$ (see \cite[Theorem 2]{byottuniqueness}), and the simple radical degree $n$ extensions of $K$ are the degree $n$ almost cyclic $H$-cyclic extensions of $K$.

It is possible to use Theorem \ref{thm:charactpure} to derive an injective correspondence from a subset of degree $n$ almost cyclic extensions of $K$ to cyclic subgroups of $K^*/(K^*)^n$. We follow the same idea as in Corollary \ref{corocorrgalois}: a simple radical extension $K(\sqrt[n]{a})$ determines a cyclic subgroup $\langle a\rangle$ of $K^*/(K^*)^n$. The following result is a generalization of \cite[Chapter III, Lemma 3]{casselsfrohlich} in which the assumption $\zeta_n\in K$ is removed.

\begin{lema}\label{lemmadiffhopf} Let $n\in\mathbb{Z}_{>0}$ and let $K$ be a field with characteristic coprime to $n$. The top fields of two degree $n$ almost cyclic extensions $K(\sqrt[n]{a})/K$ and $K(\sqrt[n]{b})/K$ with $a,b\in K$ and complement $M=K(\zeta_n)$ are $K$-isomorphic if and only if $a=b^rc^n$ with $r\in\mathbb{Z}_{\geq0}$ coprime to $n$ and $c\in M$.
\end{lema}
\begin{proof}
We can assume without loss of generality that $K(\sqrt[n]{a})/K$ and $K(\sqrt[n]{b})/K$ have degree $n$. Since $\zeta_n\in M$, the fields $M(\sqrt[n]{a})$ and $M(\sqrt[n]{b})$ are in the conditions of \cite[Chapter III, Lemma 3]{casselsfrohlich}. We obtain that $M(\sqrt[n]{a})=M(\sqrt[n]{b})$ if and only if $a=b^rc^n$ with $r\in\mathbb{Z}_{\geq0}$ coprime to $n$ and $c\in M$. Now, note that by Theorem \ref{thm:charactpure}, the fields $K(\sqrt[n]{a})/K$ and $K(\sqrt[n]{b})/K$ are almost classically Galois with complement $M$, so that their Galois closures are $M(\sqrt[n]{a})$ and $M(\sqrt[n]{b})$ respectively. This implies that each one of $M(\sqrt[n]{a})$ and $M(\sqrt[n]{b})$ contains $n$ subfields with degree $n$ over $K$ (corresponding to the $n$ roots of $x^n-a$) that are $K$-isomorphic with each other. Then, the equality $M(\sqrt[n]{a})=M(\sqrt[n]{b})$ is equivalent to $K(\sqrt[n]{a})$ and $K(\sqrt[n]{b})$ being $K$-isomorphic, and the statement follows.
\end{proof}

The main difference with respect to the Galois case is that if $\zeta_n\notin K$, the label $K(\sqrt[n]{a})$ does not determine a unique extension of $K$, but a $K$-isomorphism class of these. Hence, Lemma \ref{lemmadiffhopf} means that $a$ and $b$ generate a rank $2$ subgroup of $K^*/(K^*)^n$ if and only if they generate extensions of $K$ lying in different $K$-isomorphism classes. Thus, an extension $L=K(\alpha)$ with $\alpha^n=a\in K$ is assigned to the projection of the cyclic subgroup $\langle a\rangle$ in $K^*/(K^*)^n$, but the extensions generated by the conjugates of $\alpha$ are also sent to the same subgroup. In order to obtain an injective correspondence, we need to identify all such extensions.

On the other hand, recall that when a simple radical extension $K(\sqrt[n]{a})$ is linearly disjoint with $K(\zeta_n)$, then by Theorem \ref{thm:charactpure}, it uniquely determines an almost cyclic extension of $K$ that is $H$-cyclic, where $H$ corresponds to $K(\zeta_n)$. All together, we obtain the following:

\begin{coro}\label{coro:corrkummerHcyclic} Let $n\in\mathbb{Z}_{>0}$ and let $K$ be a field with characteristic coprime to $n$. There is an injective correspondence from the $K$-isomorphism classes of degree $n$ almost cyclic extensions $L/K$ that are $H$-cyclic, where $H$ is the almost classically Galois structure on $L/K$ corresponding to $K(\zeta_n)$, to the cyclic subgroups of $K^*/(K^*)^n$. Moreover, a cyclic subgroup $\langle a\rangle$ of $K^*/(K^*)^n$ lies in the image of this correspondence if and only if $K(\sqrt[n]{a})\cap K(\zeta_n)=K$.
\end{coro}

\subsection{The general case: radical extensions}

In this part we will provide a complete proof of Theorem \ref{maintheorem1}. We want to apply the construction in Section \ref{sectprodhg} to classes of Kummer extensions of $K$, either Galois or Hopf-Galois, to obtain Hopf-Galois structures in the compositum of those. We have already characterized simple radical extensions of $K$ as almost cyclic $H$-cyclic extensions.

The following result can be seen as a generalization of Proposition \ref{pro:root} for almost classically Galois extensions that are $H$-Kummer.

\begin{pro}\label{pro:kummerroot} Let $n\in\mathbb{Z}_{>0}$ and let $K$ be a field with characteristic coprime to $n$. Let $L/K$ be an almost abelian extension of exponent $n$ and let $H$ be an almost classically Galois structure corresponding to some Galois complement $M$ such that $\widetilde{L}/M$ is abelian. Assume that $L/K$ is $H$-Kummer and let $\{\alpha_1,\dots,\alpha_k\}$ be a generating set of $H$-eigenvectors for $L/K$. Then $M=K(\zeta_n)$, $\alpha_i^n\in K$ for all $1\leq i\leq k$, and $n$ is minimal for this property. In particular, $L/K$ is almost Kummer.
\end{pro}
\begin{proof}
By definition we have that $L\cap M=K$ and $\widetilde{L}\cong L\otimes_KM$. Then we have that $\widetilde{L}=M(\alpha_1,\dots,\alpha_k)$. Let $J=\mathrm{Gal}(\widetilde{L}/M)$, so that $H=M[J]^{G'}$ and $M\otimes_KH=M[J]$, which is the classical Galois structure on the Galois extension $\widetilde{L}/M$. Then $\{\alpha_1,\dots,\alpha_k\}$ is a generating set of $J$-eigenvectors for $\widetilde{L}/M$. Moreover, by assumption, the group $J$ has exponent $n$. Using Theorem \ref{charactkummereig}, we obtain that $\zeta_n\in M$ and $\widetilde{L}/M$ is Kummer with generators $\alpha_1,\dots,\alpha_k$. Thus $\alpha_i^n\in M$ for all $1\leq i\leq k$ and $n$ is minimal for this property, meaning that for each $1\leq l<n$ there is some $1\leq i\leq k$ such that $\alpha_i^l\notin M$. Hence $\alpha_i^n\in L\cap M=K$ and if there is some $1\leq l\leq n$ such that for every $1\leq i\leq k$ we have that $\alpha_i^l\in K$, then $\alpha_i^l\in M$ for every $1\leq i\leq k$, so $l=n$. Finally, we have that $\widetilde{L}=LK(\zeta_n)=LM$ with $K(\zeta_n)\subseteq M$, so necessarily $M=K(\zeta_n)$.
\end{proof}

As an immediate consequence, in Theorem \ref{maintheorem1}, \eqref{mainthm12} implies \eqref{mainthm11}. Moreover, as in the case of simple radical extensions, it is not necessary to assume that the extension is almost Kummer and that $M=K(\zeta_n)$, as these are implied by the hypothesis that the almost classically Galois extension is almost abelian and $H$-Kummer.

As for the converse, we will need the following natural generalization of Lemma \ref{prodgaloiseig} to this setting.

\begin{lema}\label{lem:prodeighopf} Let $L_1/K$ and $L_2/K$ be strongly disjoint almost classically Galois extensions. For each $i\in\{1,2\}$, let $H_i$ be a Hopf-Galois structure on $L_i/K$ and let $\alpha_i\in L_i$ be an $H_i$-eigenvector. Let $L=L_1L_2$ and let $H$ be the product Hopf-Galois structure of $H_1$ and $H_2$ on $L$. Then $\alpha_1\alpha_2$ is an $H$-eigenvector of $L$. In particular, the union of generating sets of $H_i$-eigenvectors for $L_i$ with $i\in\{1,2\}$ is a generating set of $H$-eigenvectors for $L$.
\end{lema}
\begin{proof}
Let $h_1\in H_1$ and $h_2\in H_2$ and consider $h=h_1h_2\in H$. Since $\alpha_i$ is an $H_i$-eigenvector of $L_i$, there are $\lambda_1(h_1),\lambda_2(h_2)\in K$ such that $h_1\cdot\alpha_1=\lambda_1(h_1)\alpha_1$ and $h_2\cdot\alpha_2=\lambda_2(h_2)\alpha_2$. From Proposition \ref{pro:prodhgstr} we obtain that $$h\cdot(\alpha_1\alpha_2)=(h_1\cdot\alpha_1)(h_2\cdot\alpha_2)=\lambda_1(h_1)\lambda_2(h_2)\alpha_1\alpha_2=\lambda(h)\alpha_1\alpha_2,$$ where $\lambda(h)=\lambda_1(h_1)\lambda_2(h_2)\in K$. Now, any element of $H$ is a sum of elements of this form, so $\alpha_1\alpha_2$ is an $H$-eigenvector. For the last sentence, note that since $1$ is an eigenvector for a Hopf-Galois structure on any extension, any $H_i$-eigenvector of $L_i$ is also an $H$-eigenvector of $L$.
\end{proof}

Now, the other implication is proved.

\begin{pro}\label{pro:radicalisalmostkummer} Let $n\in\mathbb{Z}_{>0}$, let $K$ be a field with characteristic coprime to $n$ and let $M=K(\zeta_n)$. Let $L/K$ be a strongly decomposable extension and let $\alpha_1,\dots,\alpha_k\in L$ be such that $L=K(\alpha_1,\dots,\alpha_k)$ and $K(\alpha_i)$, $K(\alpha_j)$ are strongly disjoint whenever $i\neq j$. Assume that $L\cap M=K$, $\alpha_i^n\in K$ for all $1\leq i\leq k$, and $n$ is minimal for this property. Then $L/K$ is almost Kummer of exponent $n$ with complement $M$ and $\alpha_1,\dots,\alpha_k$ are $H$-eigenvectors, where $H$ is the almost classically Galois structure corresponding to $M$.
\end{pro}
\begin{proof}
Let us call $L_i=K(\alpha_i)$ and $n_i=[L_i:K]$ for every $1\leq i\leq k$. Let us fix some such an $i$. By \cite[Lemma 3.5]{barreramora1999} there is some $m$ such that $\zeta_n^m\alpha_i^{n_i}\in K$. Since $\alpha_i^{n_i}\in L$ and $\zeta_n^m\in M$ with $L$ and $M$ $K$-linearly disjoint, necessarily $\alpha_i^{n_i}\in K$. Moreover, it is immediate that no smaller power of $\alpha_i$ belongs to $K$. In addition to this, we have that $M_i\coloneqq K(\zeta_{n_i})\subseteq M$ and $L\cap M=K$, whence $L_i\cap M_i=K$. Therefore, we can apply Theorem \ref{thm:charactpure}, which gives that $L_i/K$ is an almost cyclic extension with complement $M_i$ and $\alpha_i$ is an $H_i$-eigenvector, where $H_i$ is the almost classically Galois structure on $L_i/K$ corresponding to $M_i$. On the other hand, since $n$ is the minimal integer such that $\alpha_i^n\in K$ for all $1\leq i\leq k$, we have that $n$ is the least common multiple of $n_1,\dots,n_k$. Arguing as in the proof of \eqref{charactkummereig2} implies \eqref{charactkummereig1} in Proposition \ref{charactkummergalois}, we prove that $\zeta_n\in M$, so $M=\prod_{i=1}^kM_i$. In particular, $\widetilde{L}=LM$. By the assumption, $L_i/K$ and $L_j/K$ are strongly disjoint whenever $i\neq j$. Applying repeatedly Propositions \ref{pro:compalmostclassical} and \ref{pro:prodalmostclassical} and Lemma \ref{lem:prodeighopf}, we obtain that $L/K$ is almost classically Galois with complement $M$ and $\{\alpha_1,\dots,\alpha_k\}$ is a generating system of $H$-eigenvectors for $L/K$. Finally, by applying repeatedly Lemma \ref{lema:embedding}, the group $J=\mathrm{Gal}(\widetilde{L}/K)$ is isomorphic to the direct product of the groups $J_i=\mathrm{Gal}(\widetilde{L_i}/K)$, and hence abelian of exponent $n$. Therefore, the extension $\widetilde{L}/M$ is Kummer, so $L/K$ is almost Kummer.
\end{proof}

Theorem \ref{maintheorem1} follows immediately from Propositions \ref{pro:kummerroot} and \ref{pro:radicalisalmostkummer}. 

If in the statement of Theorem \ref{maintheorem1} we choose $k=1$, we recover Theorem \ref{thm:charactpure}. On the other hand, if we assume that $\zeta_n\in K$, then $M=K$ and $L\cap M=K$, so the strong disjointness condition translates just to linear disjointness, while \eqref{mainthm11} is just that $L/K$ is Kummer with exponent $n$. Moreover, \eqref{mainthm12} means that $L/K$ is almost classically Galois with some $K$-basis of $G$-eigenvectors, so we recover Theorem \ref{charactkummereig}.

Next, we discuss how to extend the correspondence on Corollary \ref{coro:corrkummerHcyclic} to an injective correspondence from a subset of almost Kummer extensions of $K$ to finitely generated subgroups of $K^*/(K*)^n$. First of all, an $n$-radical extension $L=K(\sqrt[n]{a_1},\dots,\sqrt[n]{a_k})$ gives rise to a finitely generated subgroup $\langle a_1,\dots,a_k\rangle$ of $K^*/(K^*)^n$. But Theorem \ref{maintheorem1} only applies to strongly decomposable extensions, so we would need that $K(\sqrt[n]{a_i})$ and $K(\sqrt[n]{a_j})$ are strongly disjoint when $i\neq j$.
Moreover, if $L\cap K(\zeta_n)=K$, $L/K$ is a strongly decomposable almost Kummer extension that is $H$-Kummer, where $H$ corresponds to $K(\zeta_n)$. We obtain:

\begin{coro}\label{coro:corrkummerhopf} Let $n\in\mathbb{Z}_{>0}$ and let $K$ be a field with characteristic coprime to $n$. There is an injective correspondence from the $K$-isomorphism classes of strongly decomposable almost Kummer extensions $L/K$ of exponent $n$ that are $H$-Kummer, where $H$ is the almost classically Galois structure on $L/K$ corresponding to $K(\zeta_n)$, to the finitely generated subgroups of $K^*/(K^*)^n$. Moreover, a finitely generated subgroup $B$ of $K^*/(K^*)^n$ is in the image of this correspondence if and only if $K(\sqrt[n]{B})$ is strongly decomposable and $K(\sqrt[n]{B})\cap K(\zeta_n)=K$.
\end{coro}

Equivalently, the assignment $B\mapsto K(\sqrt[n]{B})$ defines a bijective correspondence between:
\begin{enumerate}
    \item\label{corrhopf1} The finitely generated subgroups $B$ such that $K(\sqrt[n]{B})/K$ is strongly decomposable and $K(\sqrt[n]{B})\cap K(\zeta_n)=K$.
    \item\label{corrhopf2} The $K$-isomorphism classes of strongly decomposable almost Kummer extensions $L/K$ of exponent $n$ that are $H$-Kummer.
\end{enumerate} 

This is, thus, a generalization of the main theorem of Kummer theory. Indeed, let us assume that $\zeta_n\in K$. Take a subgroup $B$ as in \eqref{corrhopf1}. Then our assumption gives that $K(\sqrt[n]{B})/K$ is Galois. Moreover, from Proposition \ref{mordell1950th}, its simple radical subextensions are pairwise linearly disjoint. Then, they are pairwise strongly disjoint and $K(\sqrt[n]{B})/K$ is strongly decomposable. On the other hand, the assumption that $\zeta_n\in K$ also gives that the extensions as in \eqref{corrhopf2} are just the abelian extensions of $K$ with exponent $n$. Thus, this is just the correspondence in Corollary \ref{corocorrgalois}.

Finally, let us discuss the applicability of Theorem \ref{maintheorem1} to classes of extensions $L/K$. First of all, we need that any pair of intermediate fields of $L/K$ are $K$-linearly disjoint. To this end, we can use Proposition \ref{mordell1950th}. For $k=2$, this says that two simple radical extensions $L_1/K$, $L_2/K$ of degrees $n_1$, $n_2$ respectively are linearly disjoint if $K$ is totally real or $\zeta_{n_1},\zeta_{n_2}\in K$. This last situation does not represent any improvement with respect to Galois theory, as they imply that the extension $L=K(\alpha_1,\alpha_2)$ is Galois over $K$ and therefore is in the conditions of Theorem \ref{charactkummereig}.

Moreover, we also need to restrict to the simple radical extensions $L/K$ such that $L\cap K(\zeta_n)=K$. The following result gives a sufficient condition for the case $K=\mathbb{Q}$:

\begin{lema}\label{lem:inters} Let $n\in\mathbb{Z}_{>0}$. Then every radical extension $L/\mathbb{Q}$ of odd degree satisfies $L\cap\mathbb{Q}(\zeta_n)=\mathbb{Q}$.
\end{lema}
\begin{proof}
Let $m$ be an odd number and let $L=\mathbb{Q}(\sqrt[m]{a})$ with $a\in\mathbb{Q}^*/(\mathbb{Q}^*)^m$ of order $m$. Let $x\in L\cap\mathbb{Q}(\zeta_n)$. Then $\mathbb{Q}(x)$ is a subfield of $L$, so $\mathbb{Q}(x)=\mathbb{Q}(\sqrt[d]{a})$ with $d\mid m$. Moreover, this is also a subfield of $\mathbb{Q}(\zeta_n)$. Since all subextensions of a cyclotomic field are Galois, in particular so is $\mathbb{Q}(\sqrt[d]{a})/\mathbb{Q}$, and hence $d\in\{1,2\}$. Now, since $m$ is odd, it must be $d\neq2$, whence $d=1$. That is, $x\in\mathbb{Q}$ and the statement holds.
\end{proof}

\begin{rmk}\normalfont Lemma \ref{lem:inters} does not necessarily hold for extensions of $K$ with even degree. Indeed, the degree $8$ simple radical extension $\mathbb{Q}(\sqrt[8]{2})/\mathbb{Q}$ satisfies $\mathbb{Q}(\sqrt[8]{2})\cap\mathbb{Q}(\zeta_8)=\mathbb{Q}(\sqrt{2})$, since $\mathbb{Q}(\zeta_8)=\mathbb{Q}(\sqrt{2},\sqrt{-1})$.
\end{rmk}

We can use Lemma \ref{lem:inters} to find sufficient conditions for strong disjointness.

\begin{coro} Two linearly disjoint simple radical extensions $L_1/\mathbb{Q}$ and $L_2/\mathbb{Q}$ of odd degrees are strongly disjoint.
\end{coro}
\begin{proof}
Call $n_i=[L_i:K]$ for $i\in\{1,2\}$. Since $\mathbb{Q}$ is totally real, Proposition \ref{mordell1950th} gives that $L_1$ and $L_2$ are $\mathbb{Q}$-linearly disjoint. On the other hand, by Lemma \ref{lem:inters}, we have $L_i\cap \mathbb{Q}(\zeta_{n_j})=\mathbb{Q}$ for every $1\leq i,j\leq 2$. Choosing $i=j\in\{1,2\}$, we obtain that $L_i/\mathbb{Q}$ is almost classically Galois with complement $M_i=\mathbb{Q}(\zeta_{n_i})$. Finally, for $i\neq j$, we deduce that $L_1/\mathbb{Q}$ and $L_2/\mathbb{Q}$ are strongly disjoint.
\end{proof}

\section{Module structure of the ring of integers}\label{sect:modstr}

From now on, we work with the following setting: $K$ will be the fraction field of a Dedekind domain $\mathcal{O}_K$, $L$ will be an $H$-Galois extension of $K$ and $\mathcal{O}_L$ will be the integral closure of $\mathcal{O}_K$ in $L$. We also assume that $\mathcal{O}_L$ is $\mathcal{O}_K$-free. This is not always the case, see for example \cite{mackenziescheuneman}. However, it is implied for instance under the condition that $\mathcal{O}_K$ is a PID. This includes extensions of $p$-adic fields and many extensions of number fields. For the latter ones, $\mathcal{O}_K$ is a PID if and only if $K$ has class number one. In the case that $L/K$ is $H$-Kummer, we are interested in finding an $\mathcal{O}_K$-basis of $\mathfrak{A}_H$ and studying the freeness of $\mathcal{O}_L$ as module over the associated order $\mathfrak{A}_H$. 

\begin{defi}\label{defi:matrixeig} Let $L/K$ be an $H$-Galois extension and fix a $K$-basis $W=\{w_i\}_{i=1}^n$ of $H$. Assume that $B=\{\gamma_j\}_{j=1}^n$ is a $K$-basis of $H$-eigenvectors, so that for each $1\leq i,j\leq n$ there is a unique $\lambda_{ij}\in K$ such that $w_i\cdot\gamma_j=\lambda_{ij}\gamma_j$. The matrix $\Lambda_W=(\lambda_{ji})_{i,j=1}^n$ is called the \textbf{matrix of $H$-eigenvalues} with respect to $W$.
\end{defi}

It is easily checked that the matrix of eigenvalues $\Lambda_W$ is obtained from removing the zero rows of the matrix $M(H_W,L_B)$ introduced in Section \ref{sect:hgmodtheory}, where $B$ is a $K$-basis of $H$-eigenvectors for $L/K$. In fact, it is this circumstance that leads us to define the matrix of $H$-eigenvalues as in Definition \ref{defi:matrixeig}, instead of $(\lambda_{ij})_{i,j=1}^n$. We can describe completely the effect of the change of basis of $H$ on the matrix of $H$-eigenvalues as follows: if $W'=\{w_i'\}_{i=1}^n$ is another $K$-basis of $H$, then $\Lambda_{W'}=\Lambda_W P_W^{W'}$, where $P_W^{W'}$ is the matrix whose columns are the coordinates of the elements of $W'$ with respect to $W$. 

Due to the simplicity of the action, following the idea of the method in Section \ref{sect:hgmodtheory}, we can prove Theorem \ref{maintheorem2}.

\begin{proof} \textit{(of Theorem \ref{maintheorem2})}
Let $W=\{w_i\}_{i=1}^n$ be any $K$-basis of $H$, and write $w_i\cdot\gamma_j=\lambda_{ij}\gamma_j$, $\lambda_{ij}\in K$, for every $1\leq i,j\leq n$, so that $\Lambda_W=\Lambda=(\lambda_{ji})_{i,j=1}^n$.
\begin{itemize}
    \item[1.] Let us call $V=\{v_i\}_{i=1}^n$. Clearly, since $W$ is a $K$-basis of $H$, so is $V$. Let $P_W^V$ be the change basis matrix whose columns are the coordinates of elements of $V$ with respect to $W$. By definition of the elements $v_i$ we have that $P_W^V=\Omega$, so $\Lambda_V=\Lambda\Omega=\mathrm{Id}_n$. This means that $v_i\cdot\gamma_j=\delta_{ij}\gamma_j$ for every $1\leq i,j\leq n$.
    Since $B$ is an integral basis for $L/K$, $v_i\in\mathfrak{A}_H$ for every $1\leq i\leq n$. Let us check that the elements $v_i$ form a $K$-basis of $L$. For a given $h\in H$, write $h=\sum_{i=1}^nh_iv_i$, $h_i\in K$. Then $h\cdot\gamma_j=\sum_{i=1}^nh_i\delta_{ij}\gamma_j=h_j\gamma_j$. Now, we have that \begin{equation*}
        \begin{split}
            h\in\mathfrak{A}_H\Longleftrightarrow&h\cdot x\in\mathcal{O}_L\hbox{ for all }x\in\mathcal{O}_L ,
            \\\Longleftrightarrow& h\cdot\gamma_j\in\mathcal{O}_L\hbox{ for all }1\leq j\leq n,
            \\\Longleftrightarrow& h_j\gamma_j\in\mathcal{O}_L\hbox{ for all }1\leq j\leq n,
            \\\Longleftrightarrow& h_j\in\mathcal{O}_L\cap K=\mathcal{O}_K\hbox{ for all }1\leq j\leq n.
        \end{split}
    \end{equation*} Hence $\{v_i\}_{i=1}^n$ is an $\mathcal{O}_K$-basis of $\mathfrak{A}_H$. Now, we check that these are pairwise orthogonal idempotents (and since $V=\{v_i\}_{i=1}^n$ is a basis we will then obtain that it is a primitive system of idempotents). Let $1\leq i,j\leq n$. Then, for every $1\leq k\leq n$, $$(v_iv_j)\cdot\gamma_k=v_i\cdot(v_j\cdot\gamma_k)=v_i\cdot(\delta_{jk}\gamma_k)=\delta_{jk}v_i\cdot\gamma_k=\delta_{ik}\delta_{jk}\gamma_k.$$ If $i=j$, this says that $v_i^2\cdot\gamma_k=\delta_{ik}\gamma_k=v_i\cdot\gamma_k$ for every $k$, and since $\rho_H$ is injective (because $L/K$ is $H$-Galois), $v_i^2=v_i$. Otherwise, if $i\neq j$, $(v_iv_j)\cdot\gamma_k=0$ for every $k$, so again the injectivity of $\rho_H$ gives that $v_iv_j=0$.
    \item[2.] Let $\beta=\sum_{j=1}^n\beta_j\gamma_j\in\mathcal{O}_L$. Since the elements $v_i$ form an $\mathcal{O}_K$-basis of $\mathfrak{A}_H$, $\beta$ is an $\mathfrak{A}_H$-free generator of $\mathcal{O}_L$ if and only if $\{v_i\cdot\beta\}_{i=1}^n$ is an $\mathcal{O}_K$-basis of $\mathcal{O}_L$. Now, we have that $$v_i\cdot\beta=\sum_{j=1}^n\beta_j\delta_{ij}\gamma_j=\beta_i\gamma_i$$ for every $1\leq i\leq n$. Since the elements $\gamma_i$ form an $\mathcal{O}_K$-basis of $\mathcal{O}_L$, $\{v_i\cdot\beta\}_{i=1}^n$ is an $\mathcal{O}_K$-basis of $\mathcal{O}_L$ if and only if $\beta_i\in\mathcal{O}_K^*$ for every $1\leq i\leq n$. Hence any $\mathfrak{A}_H$-free generator of $\mathcal{O}_L$ is an element $\beta=\sum_{j=1}^n\beta_j\gamma_j$ such that $\beta_j\in\mathcal{O}_K^*$ for every $1\leq j\leq n$. In particular, $\mathcal{O}_L$ is $\mathfrak{A}_H$-free.
\end{itemize}
\end{proof}

\begin{rmk}\normalfont If $L/K$ is $H$-Kummer but does not admit any integral basis of $H$-eigenvectors, $\mathcal{O}_L$ is not necessarily $\mathfrak{A}_H$-free. For instance, if $L/K$ is a cyclic degree $p$ extension of $p$-adic fields and $\mathfrak{A}_{L/K}$ is the associated order in its classical Galois structure, it is known that $\mathcal{O}_L$ is not in general $\mathfrak{A}_{L/K}$-free (in fact, complete criteria for the $\mathfrak{A}_{L/K}$-freeness is known, see \cite{bertfert,bertbertfert}), while $L/K$ is Kummer.
\end{rmk}

As a first application of Theorem \ref{maintheorem2}, we can find a sufficient condition for the freeness of the ring of integers in almost cyclic extensions of $K$ as in Theorem \ref{thm:charactpure}.

\begin{pro}\label{pro:freealmostcyclic} Let $K$ be the fraction field of a Dedekind domain $\mathcal{O}_K$. Let $L=K(\alpha)$ with $\alpha\in L$ such that $\alpha^n\in K$ and no smaller power of $\alpha$ belongs to $K$. Assume that $L\cap K(\zeta_n)=K$ and that $\mathcal{O}_L=\mathcal{O}_K[\alpha]$ (in particular, $\mathcal{O}_L$ is $\mathcal{O}_K$-free). Then $\mathcal{O}_L$ is $\mathfrak{A}_H$-free, where $H$ is the almost classical Galois structure corresponding to $K(\zeta_n)$.
\end{pro}
\begin{proof}
Since $L\cap K(\zeta_n)=K$, from Theorem \ref{thm:charactpure} we obtain that $L/K$ is almost cyclic with Galois complement $K(\zeta_n)$ and that $\alpha$ is an $H$-eigenvector, where $H$ is the almost classically Galois structure corresponding to $K(\zeta_n)$. Now, from Corollary \ref{coro:primitiveeigbasis} the powers of $\alpha$ are also $H$-eigenvectors; in particular $\{1,\alpha,\dots,\alpha^{n-1}\}$ is a $K$-basis of $H$-eigenvectors of $L$. Moreover, the condition that $\mathcal{O}_L=\mathcal{O}_K[\alpha]$ ensures that this is also an $\mathcal{O}_K$-basis of $\mathcal{O}_L$. Hence, the statement follows by applying Theorem \ref{maintheorem2}.
\end{proof}

The condition that $L\cap K(\zeta_n)=K$ is in particular satisfied when $\zeta_n\in K$, which corresponds to classical Kummer extensions. In that case, Proposition \ref{pro:freealmostcyclic} becomes:

\begin{coro}\label{coro:freecyclic} Let $K$ be the fraction field of a Dedekind domain $\mathcal{O}_K$. Let $L=K(\alpha)$ with $\alpha\in L$ such that $\alpha^n\in K$ and no smaller power of $\alpha$ belongs to $K$. Assume that $\zeta_n\in K$ and that $\mathcal{O}_L=\mathcal{O}_K[\alpha]$. Then $\mathcal{O}_L$ is $\mathfrak{A}_{L/K}$-free.
\end{coro}

We know that $[K(\zeta_n):K]$ is a divisor of $\varphi(n)$, where $\varphi$ is the Euler totient function. Hence, when $n$ is a Burnside number, by definition $n$ and $\varphi(n)$ are coprime, and $L\cap K(\zeta_n)=K$. In that case, we know by Byott uniqueness theorem \cite[Theorem 2]{byottuniqueness} that the almost classically Galois structure corresponding to $K(\zeta_n)$ is the unique Hopf-Galois structure $H$ on $L/K$. These considerations lead to the following.

\begin{coro} Let $K$ be the fraction field of a Dedekind domain $\mathcal{O}_K$. Let $L=K(\alpha)$ with $\alpha\in L$ such that $\alpha^n\in K$ and no smaller power of $\alpha$ belongs to $K$. Assume that $n$ is a Burnside number and that $\mathcal{O}_L=\mathcal{O}_K[\alpha]$, and let $H$ be the unique Hopf-Galois structure on $L/K$. Then $\mathcal{O}_L$ is $\mathfrak{A}_H$-free.
\end{coro}

Next, we would like to obtain an analogue of Proposition \ref{pro:freealmostcyclic} for radical extensions, which by definition are a compositum of simple radical extensions. In addition to the restrictions in Theorem \ref{maintheorem1}, we will assume that the simple radical extensions involved are arithmetically disjoint, so that the freeness property lifts to the almost classically Galois structure on the radical extension. This means that if $L_1,\dots,L_k$ are simple radical extensions of $K$ such that $\mathcal{O}_{L_i}$ is $\mathfrak{A}_{H_i}$-free for $i\in\{1,\dots,k\}$, then $\mathcal{O}_{L_1\dots L_k}$ is $\mathfrak{A}_H$-free, where $H$ is the product Hopf-Galois structure on $L_1\dots L_k$ from $H_1,\dots,H_k$ as defined in Definition \ref{defi:prodhgstr}.

\begin{pro}\label{pro:freeraddisj} Let $K$ be the fraction field of a PID $\mathcal{O}_K$ such that $\mathrm{char}(K)$ is coprime to $n$ and let $L=K(\sqrt[n_1]{a_1},\dots,\sqrt[n_k]{a_k})$, where $a_1,\dots,a_k\in K^*$. Call $L_i=K(\sqrt[n_i]{a_i})$ for every $1\leq i\leq k$ and assume that:
\begin{enumerate}
    \item\label{freeraddisj1} $L_i\cap K(\zeta_{n_i})=K$ for every $1\leq i\leq k$ (so that $L_i/K$ is almost classically Galois).
    %\item\label{freeraddisj2} The extensions $L_i/K$ are pairwise strongly disjoint as almost classically Galois extensions.
    \item\label{freeraddisj2} $\gcd(a_1n_1,\dots,a_kn_k)=1$, where the greatest common divisor is taken in $\mathcal{O}_K$.
    \item\label{freeraddisj3} $\mathcal{O}_{L_i}=\mathcal{O}_K[\sqrt[n_i]{a_i}]$ for every $1\leq i\leq k$.
\end{enumerate}
Let $H_i$ be the almost classically Galois structure on $L_i/K$ corresponding to $K(\zeta_{n_i})$. We know from Theorem \ref{maintheorem1} that $L/K$ is almost Kummer with complement $K(\zeta_n)$ where $n=\mathrm{lcm}(n_1,\dots,n_k)$; let $H$ be the almost classically Galois structure on $L/K$ corresponding to $K(\zeta_n)$. Then $\mathfrak{A}_H=\bigotimes_{i=1}^k\mathfrak{A}_{H_i}$ and $\mathcal{O}_L$ is $\mathfrak{A}_H$-free.
\end{pro}
\begin{proof}
For each $1\leq i\leq k$, let $\alpha_i\in L$ such that $\alpha_i^{n_i}=a_i\in K$, and no smaller power of $\alpha_i$ belongs to $K$. First, note that there is no loss of generality in assuming that $a_1,\dots,a_k$ are algebraic integers, as otherwise we can multiply by the least common multiple of their denominators. Now, we note that the discriminant of a polynomial $x^n-a$ with $a\in K$ is $(-1)^{\binom{n}{2}}n^n(-a)^{n-1}$ (this is a particular case of \cite[Theorem 4]{greenfielddrucker}), and hence the condition \eqref{freeraddisj2} implies that the extensions $L_i/K$ are pairwise arithmetically disjoint. On the other hand, \eqref{freeraddisj2} holds if and only if the extensions $L_i/K$ are pairwise strongly disjoint. Indeed, since $\mathcal{O}_K$ is a PID, there are no non-trivial abelian unramified extensions of $K$ (in other words, the Hilbert Class Field of $K$ is trivial), and hence the extensions $K(\zeta_{n_j})/K$ are ramified at some prime of $K$. Then, when $i\neq j$, $K(\sqrt[n_i]{a_i})\cap K(\zeta_{n_j})=K$ if and only if $K(\sqrt[n_i]{a_i})$ and $K(\zeta_{n_j})$ do not share any ramified prime, which is equivalent to arithmetic disjointness because of the form of their discriminants.

We prove the result by a finite induction on $1\leq i<k$. For $i=1$, the statement is just Proposition \ref{pro:freealmostcyclic}. Now, assume that the statement holds for $L_1,\dots,L_i$, with $1\leq i<k$, and call $\mathcal{L}_i=L_1\dots L_i$. As already seen, the condition \eqref{freeraddisj2} ensures that $\mathcal{L}_i/K$ is strongly decomposable, so it is almost classically Galois with complement $\mathcal{M}_i=K(\zeta_{n_1})\dots K(\zeta_{n_i})=K(\zeta_{N_i})$, where $N_i=\mathrm{lcm}(n_1,\dots n_i)$. Let $\mathcal{H}_i$ be the almost classically Galois structure on $\mathcal{L}_i/K$ corresponding to $\mathcal{M}_i$. From the induction, we have that:
\begin{itemize}
    \item $\mathfrak{A}_{\mathcal{H}_i}=\bigotimes_{l=1}^i\mathfrak{A}_{H_l}$.
    \item $\mathcal{O}_{\mathcal{L}_i}$ is $\mathfrak{A}_{\mathcal{H}_i}$-free.
\end{itemize} By Proposition \ref{pro:prodalmostclassical}, $\mathcal{H}_i$ is the product Hopf-Galois structure of $H_1,\dots,H_i$ on $\mathcal{L}_i/K$. Moreover, we know that $\mathcal{L}_i\cap K(\zeta_{N_i})=K$ with $\alpha_l^{N_i}\in K$ for every $1\leq l\leq i$, and $N_i$ is minimal for that property. From Theorem \ref{maintheorem1}, we obtain that $\mathcal{L}_i/K$ is almost Kummer of exponent $N_i$ with complement $K(\zeta_{\mathrm{lcm}(n_1,\dots n_i)})$, and that elements $\alpha_l$ are $\mathcal{H}_i$-eigenvectors. On the other hand, since $L_{i+1}\cap K(\zeta_{n_{i+1}})=K$, Theorem \ref{thm:charactpure} gives that $L_{i+1}/K$ is almost cyclic with complement $K(\zeta_{n_{i+1}})$, and that an element $\alpha_{i+1}\in L_{i+1}$ such that $n_{i+1}$ is the minimal integer with $\alpha_{i+1}^{n_{i+1}}=a_{i+1}$ is an $H_{i+1}$-eigenvector. Using the hypothesis that $\mathcal{O}_{L_{i+1}}=\mathcal{O}_K[\sqrt[n_{i+1}]{a_{i+1}}]$, Proposition \ref{pro:freealmostcyclic} gives that $\mathcal{O}_{L_{i+1}}$ is $\mathfrak{A}_{H_{i+1}}$-free. 

Let $\mathcal{L}_{i+1}=\mathcal{L}_iL_{i+1}=L_1\dots L_{i+1}$. It follows from \eqref{freeraddisj2} that the extensions $\mathcal{L}_i/K$ and $L_{i+1}/K$ are strongly disjoint. By Proposition \ref{pro:compalmostclassical}, $\mathcal{L}_{i+1}/K$ is almost classically Galois with complement $\mathcal{M}_{i+1}=K(\zeta_{\mathrm{lcm}(n_1,\dots n_{i+1})})$. Hence, Proposition \ref{pro:prodalmostclassical} gives that the almost classically Galois structure $\mathcal{H}_{i+1}$ on $\mathcal{L}_{i+1}/K$ corresponding to $\mathcal{M}_{i+1}$ is just the product Hopf-Galois structure of $\mathcal{H}_i$ and $H_{i+1}$ on $\mathcal{L}_{i+1}/K$. Therefore from Corollary \ref{coro:prodhgmstr} it follows that the statement is satisfied for $L_1,\dots,L_{i+1}$.
\end{proof}

\subsection{The case of number fields}

Let $L/K$ be an $n$-radical extension of number fields such that $L\cap K(\zeta_n)=K$, so that $L/K$ is almost Kummer. Let $H$ be its almost classically Galois structure corresponding to $K(\zeta_n)$. Recall that the number fields $K$ such that $\mathcal{O}_K$ is a PID are just the fields with class number one. In this situation, Proposition \ref{pro:freeraddisj} gives sufficient conditions for $\mathcal{O}_L$ being $\mathfrak{A}_H$-free.

Note that in all this discussion we do not need any assumption on the ramification of $L/K$. In particular, this includes tamely ramified extensions of $K$. However, the existing results in that case are characterizations of the freeness for subclasses of these extensions. Namely, Del Corso and Rossi \cite[Theorem 11]{delcorsorossi2013} characterized the existence of a normal integral basis (in other words, the freeness of the ring of integers over the associated order in the classical Galois structure) for tamely ramified Kummer extensions. Moreover, Truman \cite[Theorem 5.5]{truman2020} characterized the freeness over the associated order in the unique Hopf-Galois structure for tamely ramified simple radical extensions of prime degree. However, Proposition \ref{pro:freeraddisj} involves the freeness of the ring of integers in a broad class of wildly ramified radical extensions of number fields.

Let us derive some interesting particular cases of this one and the other results at Section \ref{sect:modstr}. For a simple radical extension $L=K(\sqrt[n]{a})$, the condition that $\mathcal{O}_L=\mathcal{O}_K[\sqrt[n]{a}]$ means that an element $\alpha\in\mathcal{O}_L$ with $\alpha^n=a$ generates a power integral basis of $L/K$. When $K=\mathbb{Q}$ and $a$ is square-free, Gassert \cite[Theorem 1.1]{gassert} proved that this is equivalent to $a^p\not\equiv a\,(\mathrm{mod}\,p^2)$ for every prime divisor $p$ of $n$. Applying Proposition \ref{pro:freealmostcyclic}, we conclude:

\begin{coro} Let $L=\mathbb{Q}(\sqrt[n]{a})$, with $a\in\mathbb{Z}$ square-free. Assume that $L\cap\mathbb{Q}(\zeta_n)=\mathbb{Q}$ and 
$a^p\not\equiv a\,(\mathrm{mod}\,p^2)$ for every prime divisor $p$ of $n$. Then, $\mathcal{O}_L$ is $\mathfrak{A}_H$-free, where $H$ is the almost classically Galois structure corresponding to $\mathbb{Q}(\zeta_n)$.
\end{coro}

Likewise, by specializing Proposition \ref{pro:freeraddisj} suitably we obtain:

\begin{coro} Let $L=\mathbb{Q}(\sqrt[n_1]{a_1},\dots,\sqrt[n_k]{a_k})$, with $a_1,\dots,a_k\in\mathbb{Z}$ square-free such that $a_i^{p_i}\not\equiv a_i\,(\mathrm{mod}\,p_i^2)$ for every prime divisor $p_i$ of $n_i$ and $\gcd(a_1n_1,\dots,a_kn_k)=1$. Then, $\mathcal{O}_L$ is $\mathfrak{A}_H$-free, where $H$ is the almost classically Galois structure corresponding to $\mathbb{Q}(\zeta_{n_1\dots n_k})$.
\end{coro}

Next, consider a simple radical extension $L=K(\sqrt[p]{a})$ with $p$ prime. If $K=\mathbb{Q}(\zeta_p)$, Smith \cite{smith2020} proved that $\mathcal{O}_L=\mathcal{O}_K[\sqrt[p]{a}]$ if and only if the ideal generated by $a$ in $\mathbb{Z}[\zeta_p]$ is square-free and $a^p\not\equiv a\,(\mathrm{mod}\,(1-\zeta_p)^2)$. From Corollary \ref{coro:freecyclic}, we deduce:

\begin{coro} Let $K=\mathbb{Q}(\zeta_p)$ and let $L=K(\sqrt[p]{a})$, where $a\in K$. Assume that the ideal $\langle a\rangle$ of $\mathbb{Z}[\zeta_p]$ is square-free and that $a^p\not\equiv a\,(\mathrm{mod}\,(1-\zeta_p)^2)$. Then $\mathcal{O}_L$ is $\mathfrak{A}_{L/K}$-free.
\end{coro}

\subsection{The case of $p$-adic fields}

In this part we consider simple radical degree $n$ extensions $L/K$ of $p$-adic fields that are linearly disjoint with $K(\zeta_n)$. As in the case of number fields, we know sufficient conditions for the freeness of $\mathcal{O}_L$ as an $\mathfrak{A}_H$-module, where $H$ is the almost classically Galois structure on $L/K$ corresponding to $M\coloneqq K(\zeta_n)$. For simplicity, we will assume that $L/K$ is totally ramified.

First, assume that $p\nmid n$, in which case $L/K$ is tamely ramified. By definition, $H=\widetilde{L}[J]^G$, where $J$ is the Galois group of the cyclic extension $\widetilde{L}/M$, and in particular abelian. Hence $H$ is commutative, and applying \cite[Theorem 5.3]{truman2018} gives that $\mathcal{O}_L$ is $\mathfrak{A}_H$-free. Hence, from now on we will be preferably interested in wildly ramified extensions, for which $p\mid n$.

Recall that totally ramified extensions of $p$-adic fields are those for which there is a uniformizer which is a root of some $p$-Eisenstein polynomial. If we assume that such a polynomial is in addition a radical one, we obtain the following:

\begin{pro}\label{pro:freepadictotram} Let $K$ be a $p$-adic field and let $L=K(\sqrt[n]{a})$ with $a\in K$ and $v_K(a)=1$. Suppose in addition that $L\cap K(\zeta_n)=K$. Then $\mathcal{O}_L$ is $\mathfrak{A}_H$-free, where $H$ is the almost classically Galois structure on $L/K$ corresponding to $K(\zeta_n)$.
\end{pro}
\begin{proof}
Let $\alpha\in L$ with $\alpha^n=a$. Since $v_K(a)=1$, $\alpha$ is a root of a $p$-Eisenstein polynomial and $L/K$ is totally ramified. Then, we have that $v_L(\alpha)=\frac{v_L(a)}{n}=1$ and hence $\alpha$ is a uniformizer of $L$. Thus, $\mathcal{O}_L=\mathcal{O}_K[\sqrt[n]{a}]$, so we can apply Proposition \ref{pro:freealmostcyclic} to obtain that $\mathcal{O}_L$ is $\mathfrak{A}_H$-free.
\end{proof}

Again, if $n$ is a Burnside number (in particular, if $n$ is prime), the condition that $L\cap K(\zeta_n)=K$ is automatically satisfied.

\subsubsection*{Maximally ramified extensions}

An interesting particular case of the previous result is when $L/K$ has degree $p$ and its normal closure is maximally ramified. This means that the ramification jump of the normal closure (where the only jump we take into account is from the cyclic group of order $p$ to the trivial group, see for instance \cite[\textsection 1.2]{berge1978}) is as high as possible. Our objective is to prove the following.

\begin{pro}\label{pro:freemaxram} Let $L/K$ be a degree $p$ extension of $p$-adic fields whose normal closure is maximally ramified over $K$. Then $L/K$ is almost classically Galois with complement $K(\zeta_p)$ and $\mathcal{O}_L$ is $\mathfrak{A}_H$-free, where $H$ is the almost classically Galois structure on $L/K$ corresponding to $K(\zeta_p)$.
\end{pro}

We will need some preparations.

Let $E/F$ be a cyclic degree $p$ extension of $p$-adic fields. We assume that $E/F$ is ramified, in which case it is totally ramified and it has a unique ramification jump $t$. Let $e=e(F/\mathbb{Q}_p)$. Then, it is known that $1\leq t\leq\frac{pe}{p-1}$ (see \cite[Chapter IV, Proposition 2]{serre}). The condition that $E/F$ is maximally ramified is that $t=\frac{pe}{p-1}$. It is also known that if $E/F$ is maximally ramified then $F$ contains the $p$-th roots of unity and $E=F(\sqrt[p]{\pi_F})$ (see for instance \cite[Proposition 6.3]{delcorsoferrilombardo}).

Now, let $L/K$ be a degree $p$ extension of $p$-adic fields and let $\widetilde{L}$ be its normal closure. Let $G=\mathrm{Gal}(\widetilde{L}/K)$ and consider the chain of ramification groups $\{G_i\}_{i=1}^{\infty}$ for $\widetilde{L}/K$. Assume that $\widetilde{L}/K$ is totally ramified. By \cite[Chapter IV, Corollary 5]{serre}, $G=G_0$ is solvable, and hence by \cite[(7.5)]{childs}, $L/K$ is Hopf-Galois. Since $p$ is Burnside, Byott's uniqueness theorem \cite[Theorem 2]{byottuniqueness} gives that $L/K$ admits a unique Hopf-Galois structure and it is almost classically Galois. On the other hand, \cite[Chapter IV, Corollaries 3, 4]{serre} give that $G_1$ is cyclic of order $p$ and $G=G_1\rtimes C$, where $C$ is cyclic of order $r$ coprime to $p$. Let us establish a presentation $$G=\langle\sigma,\tau\,|\,\sigma^p=\tau^r=1,\,\tau\sigma\tau=\sigma^g\rangle,$$ where $g\in\mathbb{Z}$ has order $r$ modulo $p$, such that $G_1=\langle\sigma\rangle$ and $C=\langle\tau\rangle$. We can assume without loss of generality that $L=\widetilde{L}^C$.

Moreover, note that there is a unique $t\geq1$ such that $G_t\cong C_p$ and $G_{t+1}=\{1\}$. In analogy with the cyclic case, $t$ is also called the ramification jump of the extension $\widetilde{L}/K$. From \cite[Chapter IV, Proposition 2]{serre} we deduce that $t$ is the unique ramification jump of the cyclic degree $p$ extension $\widetilde{L}/M$. Since $M/K$ is totally ramified, the ramification index of $M/\mathbb{Q}_p$ is $re$, where $e=e(K/\mathbb{Q}_p)$. Then, we have that $$1\leq t\leq\frac{rpe}{p-1},$$ and $\widetilde{L}/K$ is maximally ramified if and only if $t=\frac{rpe}{p-1}$.

\begin{lema} Let $L/K$ be a degree $p$ extension of $p$-adic fields whose normal closure is maximally ramified over $K$. Then $L/K$ is almost classically Galois with complement $K(\zeta_p)$ and there is $\alpha\in\mathcal{O}_L$ such that $v_L(\alpha)=1$ and $\alpha^p\in\mathcal{O}_K$.
\end{lema}
\begin{proof}
The fact that $L/K$ is almost classically Galois follows from the previous discussion due to the hypothesis that $L/K$ is ramified (which for degree $p$ extensions is equivalent to total ramification); call $M$ its complement. Let $\zeta_p$ be a primitive $p$-th root of the unity and let $\pi_M$ be a uniformizer of $M$. Since $\widetilde{L}/M$ is maximally ramified, $\zeta_p\in M$ and $\widetilde{L}=M(\sqrt[p]{\pi_M})$, that is, there is $\gamma\in\mathcal{O}_{\widetilde{L}}$ with $v_{\widetilde{L}}(\gamma)=1$ and $\gamma^p\in\mathcal{O}_M$ such that $\widetilde{L}=M(\gamma)$.

Let $\alpha=N_{\widetilde{L}/L}(\gamma)\in L$. Then $v_{\widetilde{L}}(\alpha)=rv_L(\alpha)$, and on the other hand, using that $\gamma$ is a uniformizer of $\widetilde{L}$, we have $v_{\widetilde{L}}(\alpha)=r$, so $v_L(\alpha)=1$. In particular, $\alpha$ is a primitive element of $L/K$. Moreover, $\alpha^p=N_{\widetilde{L}/L}(\gamma^p)\in L\cap M=K$, and it is clear that $\alpha$ is an algebraic integer, so $\alpha^p\in\mathcal{O}_K$. Finally, we prove that $M=K(\zeta_p)$. Since $\alpha^p\in K$, the normal closure of $L/K$ is $\widetilde{L}=K(\alpha,\zeta_p)=LK(\zeta_p)$, and since we have also that $\widetilde{L}=LM$ with $L$ and $M$ $K$-linearly disjoint and $K(\zeta_p)\subseteq M$, $M=K(\zeta_p)$.
\end{proof}

We have proved that a degree $p$ extension of $p$-adic fields $L/K$ with maximally ramified normal closure is almost classically Galois with complement $K(\zeta_p)$, and in particular $L\cap K(\zeta_p)=K$. Let $\alpha\in L$ be as in the previous lemma and call $a\coloneqq\alpha^p\in K$. Then $L=K(\sqrt[p]{a})$ and $v_L(\alpha)=1$, that is, $v_K(a)=1$. Hence we are under the hypotheses of Proposition \ref{pro:freepadictotram}, and then Proposition \ref{pro:freemaxram} is established.

\section*{Acknowledgements}

The author is thankful to Ilaria Del Corso, Paul Truman and Lorenzo Stefanello for their insightful comments. The author also wants to thank the referee for their helpful comments and suggestions to improve the quality of this paper. This work was supported by Czech Science Foundation, grant 21-00420M, by Charles University Research Centre program UNCE/SCI/022 and by the grant PID2022-136944NB-I00 (Ministerio de Ciencia e Innovación).

\printbibliography

%\Addresses

\end{document}